\documentclass{amsart}
\usepackage{amssymb,accents}
\usepackage{tikz}

\usepackage{amsthm}

\usepackage{color}
\usepackage{mathtools}

\usepackage{tikz}

\newtheorem{thm}{Theorem}[section]
\newtheorem{cor}[thm]{Corollary}

\newtheorem{prop}[thm]{Proposition}
\newtheorem{lem}[thm]{Lemma}

\theoremstyle{definition}
\newtheorem*{excont}{Example \continuation}
\newcommand{\continuation}{(continued)}
\newenvironment{continueexample}[1]
 {\renewcommand{\continuation}{\ref{#1}}\excont[continued]}
 {\endexcont}

\newtheorem{example}[thm]{Example}
\newtheorem{definition}[thm]{Definition}

\newcommand{\hlight}[1]{{\color{blue} #1}}

\newcommand{\supth}{{}^{\rm{th}}}

\newcommand{\ncirc}[1]{++(#1,0) circle (5pt)}

\DeclareMathOperator{\lspan}{\operatorname{span}}
\DeclareMathOperator{\codim}{\operatorname{codim}}

\DeclareMathOperator{\End}{\operatorname{End}}
\DeclareMathOperator{\Ann}{\operatorname{Ann}}
\DeclareMathOperator{\Ker}{\operatorname{Ker}}

\newcommand{\cc}{c}

\newcommand{\hcc}{{\hat{c}}}
\newcommand{\cof}[1]{\cc(#1)}

\newcommand{\e}{\mathrm{e}}
\newcommand{\fA}{\mathfrak{A}}

\newcommand{\rE}{\mathrm{E}}

\newcommand{\bev}[2]{\operatorname{ev}(#1,#2)}
\newcommand{\hbev}[2]{\widehat{\operatorname{ev}(#1,#2)}}

\newcommand{\al}{\alpha}
\newcommand{\alq}{\alpha^{(q)}}
\newcommand{\bq}{b^{(q)}}

\newcommand{\bH}{\boldsymbol{H}}

\renewcommand{\boxdot}{{\ \clap{\raise0.25ex\hbox{$\bullet$}}\clap{$\square$}\ }}

\newcommand{\la}{\left\langle}
\newcommand{\ra}{\right\rangle}

\newcommand{\lp}{\left(}
\newcommand{\rp}{\right)}

\newcommand{\Gr}{{\operatorname{Gr}}}

\newcommand{\Grad}{{\operatorname{Gr}^{\mathrm{ad}}}}

\newcommand{\cA}{\mathcal{A}}

\newcommand{\cC}{\mathbb{D}}
\newcommand{\cG}{\mathcal{G}}
\newcommand{\cJ}{\mathcal{J}}
\newcommand{\cK}{\mathcal{K}}

\newcommand{\cL}{\mathcal{L}}

\newcommand{\cS}{\mathcal{S}}
\newcommand{\cP}{\mathcal{P}}

\newcommand{\tT}{\tilde{T}}

\newcommand{\hpi}{\hat{\pi}}
\newcommand{\hPhi}{\hat{\Phi}}

\newcommand{\tC}{\tilde{C}}

\newcommand{\piflat}{\pi^\flat}
\newcommand{\rEflat}{\rE^\flat}
\newcommand{\hrE}{\mathrm{\hat{E}}}
\newcommand{\pinat}{\pi^\natural}

\newcommand{\sigsharp}{\sigma^\sharp}

\newcommand{\cD}{\mathcal{D}}
\newcommand{\cR}{\mathcal{R}}

\newcommand{\Wr}{\operatorname{Wr}}
\newcommand{\N}{\mathbb{N}}
\newcommand{\Q}{\mathbb{Q}}
\newcommand{\Nz}{{\N_0}}
\newcommand{\Z}{\mathbb{Z}}
\newcommand{\C}{\mathbb{C}}
\newcommand{\R}{\mathbb{R}}

\newcommand{\bt}{\boldsymbol{t}}

\newcommand{\balpha}{\boldsymbol{\alpha}}
\newcommand{\balphaq}{\boldsymbol{\alpha^{(q)}}}

\newcommand{\sS}{\mathsf{S}}

\newcommand{\Ks}{{\accentset{*}{K}}}
\newcommand{\Ws}{{\accentset{*}{W}}}
\newcommand{\Kslam}{\Ks{}^{(\lambda)}}
\newcommand{\Klam}{K^{(\lambda)}}

\newcommand{\upslam}{\upsilon^{(\lambda)}}

\newcommand{\Knat}{K^\natural}

\newcommand{\KC}{K_C}

\newcommand{\KCp}{K_{C'}}
\newcommand{\Hlam}{H^{(\lambda)}}
\newcommand{\Clam}{C^{(\lambda)}}

\newcommand{\Cslam}{\accentset{*}{C}{}^{(\lambda)}}

\newcommand{\hClam}{\hat{C}^{(\lambda)}}
\newcommand{\chilam}{\chi^{(\lambda)}}
\newcommand{\Ilam}{\mathcal{I}^{(\lambda)}}
\newcommand{\Slam}{S^{(\lambda)}}
\newcommand{\wlam}{w^{(\lambda)}}
\newcommand{\Wlam}{W^{(\lambda)}}
\newcommand{\hWlam}{\hW^{(\lambda)}}

\newcommand{\taulam}{\tau^{(\lambda)}}
\newcommand{\ulam}{u^{(\lambda)}}

\newcommand{\Glam}{G^{(\lambda)}}
\newcommand{\Gamlam}{\Gamma^{(\lambda)}}
\newcommand{\cGlam}{\cG^{(\lambda)}}
\newcommand{\glam}{g^{(\lambda)}}
\newcommand{\gamlam}{\gamma^{(\lambda)}}

\newcommand{\Tlam}{T^{(\lambda)}}
\newcommand{\tTlam}{\tT^{(\lambda)}}
\newcommand{\Llam}{L^{(\lambda)}}
\newcommand{\Rlam}{R^{(\lambda)}}

\newcommand{\Uslam}{\accentset{*}{U}{}^{(\lambda)}}

\newcommand{\Wslam}{\Ws^{(\lambda)}}

\newcommand{\nulam}{\nu^{(\lambda)}}

\newcommand{\kaplam}{\kappa^{(\lambda)}}
\newcommand{\kaplamn}[1]{\kappa^{(\lambda)}_{#1}}
\newcommand{\cKlam}{\mathcal{K}^{(\lambda)}}
\newcommand{\Mlam}{{\mathcal{M}^{(\lambda)}}}
\newcommand{\cRlam}{\cR^{(\lambda)}}
\newcommand{\cDlam}{\cD^{(\lambda)}}
\newcommand{\cDslam}{\accentset{*}{\cD}^{(\lambda)}}
\newcommand{\cSlam}{{\cS^{(\lambda)}}}

\newcommand{\cRs}{\accentset{*}{\cR}}
\newcommand{\cJslam}{\accentset{*}{\cJ}{}^{(\lambda)}}
\newcommand{\cRslam}{\cRs{}^{(\lambda)}}
\newcommand{\cSslam}{\accentset{*}{\cS}^{(\lambda)}}
\newcommand{\fSlam}{\fS^{(\lambda)}}

\newcommand{\fAlam}{\fA^{(\lambda)}}

\newcommand{\Thetlam}{\Theta^{(\lambda)}}
\newcommand{\siglam}{\sigma^{(\lambda)}}

\newcommand{\cAlam}{\cA^{(\lambda)}}
\newcommand{\cAslam}{\accentset{*}{\cA}^{(\lambda)}}

\newcommand{\Flam}{\mathcal{F}^{(\lambda)}}

\newcommand{\Jlam}{\cJ^{(\lambda)}}

\newcommand{\Philam}{\Phi^{(\lambda)}}
\newcommand{\Psilam}{\Psi^{(\lambda)}}

\newcommand{\E}{\mathrm{E}}

\newcommand{\fS}{\mathfrak{S}}

\newcommand{\hc}{\hat{c}}

\newcommand{\hC}{\hat{C}}
\newcommand{\hW}{\hat{W}}

\title{The Adelic Grassmannian and Exceptional Hermite Polynomials}

\author{Alex Kasman} \address{ Dept. of Mathematics, College of
Charleston, Charleston SC 29424, USA}

\author{Robert Milson} \address{Dept. of Mathematics and Statistics,
Dalhousie University, Halifax NS B3H 3J5, Canada.}

\email{kasmana@cofc.edu, rmilson@dal.ca}

\begin{document}
\begin{abstract}
  It is shown that when dependence on the second flow of the KP
  hierarchy is added, the resulting semi-stationary wave
  function of certain points in George Wilson's adelic Grassmannian
  are generating functions of the exceptional Hermite orthogonal
  polynomials.  This surprising correspondence between different
  mathematical objects that were not previously known to be so closely
  related is interesting in its own right, but also proves useful in
  two ways: it leads to new algorithms for effectively computing the
  associated differential and difference operators and it also answers
  some open questions about them.
\end{abstract}

\maketitle

\section{Introduction}

Suppose that operators $L$ and $\Lambda$ share an eigenfunction
$\psi(x,z)$ so that
\begin{equation}
L\psi=p(z)\psi\qquad\hbox{and}\qquad \Lambda\psi=\pi(x)\psi.\label{eq:bisp}
\end{equation}
where $L$ is an operator
acting on functions of $x$ that is independent of $z$, $\Lambda$ is an
operator on functions of $z$ that is independent of $x$, and the
eigenvalues $p(z)$ and $\pi(x)$ are both non-constant functions.  In
this case we say that the operators $L$ and $\Lambda$ are
\textit{bispectral} and that $(L,\Lambda,\psi)$ is a \textit{bispectral triple.}

The term ``bispectrality'' was used to describe this situation in
\cite{DG} where the authors identified all bispectral
 Schr\"odinger operators of the form $L=\partial_x^2+u(x)$ sharing an
 eigenfunction with an ordinary differential operator $\Lambda$ in $z$.
Bispectrality has since been considered in much greater
generality, allowing $L$ and $\Lambda$ to be any sorts of operators
acting on functions of variables that can be either scalar or vector
valued and either discrete or continuous.

A fundamental breakthrough in the study of bispectral operators was
George Wilson's 1993 paper \cite{wilson} in which he made use of the
Sato Grassmannian $\Gr$ which was developed for producing solutions to the
KP hierarchy \cite{Sato,SW}.  That construction normally associates to
a point $W\in\Gr$ a
pseudo-differential operator $\cL_W$ depending on time variables
$t_i$ for $i=1,2,3\ldots$ and a ``wave\ function''
$\psi_W(t_1,t_2,\ldots,z)=(1+O(z^{-1}))\textup{exp}\sum_{i=1}^{\infty}t_iz^i$ satisfying the equations
\begin{equation}
\cL_W\psi_W=z\psi_W\qquad\hbox{and}\qquad
{\partial_{t_i}}\psi_W=(\cL_W^i)_+\psi_W\label{eqn:KPhierarchy}
\end{equation}
which together are equivalent to a hierarchy of nonlinear evolution
equations for the coefficients of $\cL_W$. The connection to
bispectrality is most apparent when one ``turns off'' all but one of
the time variables. By setting $t_1=x$ and $t_i=0$ for $i>1$ the wave\
function becomes a function of only the variables $x$ and $z$ as in
the definition of bispectrality.  Wilson showed that if $W$ is in the
``adelic Grassmannian'' $\Grad\subset\Gr$ then there is a non-trivial
ring of \textit{ordinary} differential operators in $x$ having
$\psi_W$ as an eigenfunction \textit{and} there exists another point
$\beta(W)\in\Grad$ such that $\psi_W(x,z)=\psi_{\beta(W)}(z,x)$.
Wilson generalized the notion of a bispectral triple to allow for $L$
and $\Lambda$ to be elements of commutative algebras, rather than
fixed operators, and showed that (up to trivial renormalizations)
$\Grad$ is the moduli space of bispectral ordinary differential
operators that commute with operators of relatively prime order
\cite{wilson}.

Following Wilson's seminal paper, it is now recognized that it is more
natural to define a bispectral triple as two commutative algebras of
operators and a common eigenfunction.  Furthermore, nearly all papers
on this subject now produce bispectral algebras using versions of
Wilson's construction that have been suitably modified to different
settings.  For example, to consider the case of differential operators
that do not commute with operators of relatively prime order the
scalar eigenfunction is replaced with a vector eigenfunction and to
study bispectral difference operators the (pseudo)-differential
operators are replaced with infinite matrices (e.g. see \cite{KRAiry}
and \cite{GY}).

Classical orthogonal polynomials are families of orthogonal
polynomials that are the eigenfunction of a second-order
Sturm-Liouville eigenvalue problem.  As such, they may be considered
as the eigenfunctions of a differential-difference bispectral triple,
where the 3-term recurrence plays the role of the differential
equation in spectral parameter.  This idea can be naturally connected
to the adelic grassmannian \cite{HI}.  See also \cite{GH,iliev} and the
references therein for an application of such ideas to Krall
polynomials, and the Askey-Wilson scheme.

Exceptional orthogonal polynomials\cite{GKM,OS1} generalize classical
orthogonal polynomials, because they are the eigenfunctions of a
second-order Sturm-Liouville problem, but fall outside the
Askey-Wilson scheme by allowing for polynomial sequences that omit a
finite number of ``exceptional'' degrees.  This relaxed assumption
implies that exceptional polynomials cannot satisfy a 3-term
recurrence relations.  Indeed, unlike the differential-differential
bispectral problem investigated by Duistermaat and Gr\"unbaum
\cite{DG}, the dual eigenvalue problem for exceptional Hermite
polynomials consists of an algebra of commuting difference operators
\cite{GKKM}.  A similar situation seems to hold for the case of
exceptional Laguerre and Jacobi polynomials \cite{bondunstev,duran}
and for discrete exceptional polynomials \cite{duran2,OS2}.  This
observation suggests that the ensemble consisting of (i) a family of
exceptional orthogonal polynomials, (ii) the corresponding exceptional
second-order operator, and (iii) higher order recurrences should also
be regarded as an instance of a differential-difference bispectral
triple.

In the case of exceptional Hermite polynomials, the second-order
exceptional operator in question is known to have trivial monodromy
\cite{oblomkov}.  It is also known that every exceptional operator is
related by a Darboux transformation to a classical operator
\cite{GFGM}.  All of this is a further indication that bispectrality
should be a key concept in the theory of exceptional polynomials.

The present paper grew out of an investigation of questions
surrounding the bispectrality of the exceptional Hermite orthogonal
polynomials \cite{duran,GGM,GKKM}.  Since the differential operators
that have these as eigenfunctions all have even degree and since the
operator in the other variable is a difference operator, one might
expect that Wilson's construction should be suitably modified to
address these questions.  However, the most surprising result to be
presented below is the fact that essentially no modification is
needed; the exceptional Hermite orthogonal polynomials were already
present (but unnoticed) in Wilson's original construction.  This
observation turns out to be quite useful, greatly simplifying the
construction of the exceptional Hermite orthogonal polynomials and the
associated operators, and providing answers to some open questions
surrounding them.

 The organization of the paper is as follows.
\begin{itemize}
\item Section \ref{sect:SF} collects some background material on
  partitions, Maya diagrams and  Schur functions.
\item Section \ref{sect:grad} is a quick review of $\Grad$, Wilson's
  adelic grassmannian and of the bispectral involution.
\item Section \ref{sect:XOP} reviews classical and exceptional Hermite
  polynomials and shows how these objects are naturally associated
  with certain points in $\Grad$.  In particular,
  Theorem~\ref{thm:generatingxops} shows that the wave function
  corresponding to self-dual points $\Wlam \in \Grad$ labelled by a
  partition $\lambda$ serve as generating functions for the family of
  exceptional Hermite orthogonal polynomials (cf.~\cite{GGM,GKKM})
  associated with the same partition. The only modification needed to
  Wilson's original construction is that rather than setting all time
  variables $t_i$ with $i>1$ equal to zero, one must instead only set
  time variables with index $i>2$ equal to zero.  The second time
  variable $y=t_2$ plays the role of a scaling parameter for the
  exceptional orthogonal polynomials.    
\item Theorem~\ref{thm:Hlamlincomb} gives a useful formula for the
  exceptional Hermite polynomials as linear combinations of the
  classical Hermite polynomials with coefficients derived from wave
  functions in $\Grad$.  This formulation has a significant
  computational and conceptual advantage over the usual formulation in
  terms of Wronskians.
\item Section~\ref{sect:1ptcond}: The exceptional Hermite
  orthogonal polynomials are known to be annihilated by point
  supported distributions, however it was not previously known how to
  determine which distributions annihilated a given instance of the
  exceptional polynomials.  In the context of $\Grad$, this question
  is answered easily using Wilson's bispectral involution
  $W\mapsto \beta(W)$.
\item Section~\ref{sect:bispect} introduces the bispectral triple
  associated with a given family of exceptional Hermite polynomials.
  In this context, it is natural to introduce and to study
  non-commutative stabilizer algebras corresponding to the points in
  $\Grad$.  The bispectral involution defines an anti-isomorphism
  between two such algebras.  The eigenvalue relations engendered by
  the bispectral triple can then be conveniently constructed as the
  restriction of this anti-isomorphism to the commutative subalgebras
  corresponding to the eigenvalues.
\item Exceptional Hermite polynomials are known to satisfy lowering
  relations \cite{GGMM} and higher-order recurrence relations
  \cite{GKKM}.  Section~\ref{sect:lowop}: the algebra of lowering
  operators is naturally isomorphic to the stabilizer algebra of
  $\Wlam$.  The corresponding algebraic structure has can be easily
  understood in terms of certain combinatorial properties of the
  partition $\lambda$.  Section~\ref{sect:rr}: The ring of the
  corresponding higher-order Jacobi operators can then be seen to be
  the stabilizer algebra of $\beta(\hWlam)$, the bispectral dual of
  the curve generated by $\Wlam$ under the second KP flow.  The
  $\Grad$-based construction allows for a significant computational
  advantage in determining the form of these exceptional Jacobi
  operators.
\item Section~\ref{sect:algex} collects some examples and details of
  calculations related the intertwining relations, lowering operators
  and exceptional recurrence relations.
\end{itemize}




\subsection{Conventions and Notations}\label{sec:conventions}

Let $\cP$ denote the ring of complex valued univariate polynomials
regarded as mappings without reference to any particular variables.
If $U\subset \cP$ is a polynomial subspace and $x$ is an
indeterminate, we will employ the notation $U(x)$ to denote a subspace
of the corresponding $\C[x] \simeq \cP$.
Most of the functions encountered below will depend on several
variables, and so we will use $\partial_x$ to indicate denote the
elementary partial-derivative operator
$\displaystyle\frac{\partial}{\partial x}$.
The symbol $\Wr_x$ denotes the usual Wronskian determinant with
respect to the indeterminate $x$. The symbol $\Wr$, without any
subscript, denotes the Wronskian taken with respect to the first
argument.  The symbol $\sS$ denotes the unit shift operator.  If
$f_n,\; n\in \Nz$ is a sequence indexed by an indeterminate $n$, we
will write $\sS_n f_n := f_{n+1}$.

For later convenience, we also adopt the following unusual convention.
When a differential operator is constructed by the substitution of an
elementary partial-derivative operator into a multi-variable function,
then it is understood that the derivatives all appear to the right of
all other variables --- regardless of whether they commute as
operators.  In contrast, when substituting a differential operator
into a univariate polynomial, it is understood that powers of the
operator are computed through composition.  For example, let
$\pi(x,y,z)=x^2y^{-1}z-xy^3z^4$.  Then, according to this convention
we have
\begin{align*}
  \pi(\partial_z,y,z)&=y^{-1}z\partial_z^2-y^3z^4\partial_z\\
  \pi(x,y,\partial_x)&=x^2y^{-1}\partial_x-xy^3\partial_x^4.
\end{align*}
In particular, applying either of those differential operators to the
function $e^{xz}$ results in $\pi(x,y,z)e^{xz}$.  On the other hand,
if $\gamma(z)=z^3$ then we define
\[\gamma(z\partial_z)=(z\partial_z)\circ (z\partial_z) \circ (z\partial_z)
  = z^3\partial_z^3+3z^2\partial_z^2+z\partial_z.\]

The following notations are all rigorously introduced as needed later
in the text, but are briefly summarized here for the reader's
convenience.  The symbol $\N= \{ 1,2,\ldots \}$ is the set of natural
numbers, while $\Nz = \{0,1,2,\ldots\}$ is the set of non-negative
integers.  The usual univariate classical Hermite polynomials will be
denoted by $h_n(x),\; n\in \Nz$ while $H_n(x,y),\; n\in \Nz$ are the
same polynomials converted to a bivariate form through the inclusion
of a scaling parameter.  Similarly, $h^{(\lambda)}_n(x)$ will denote
the exceptional Hermite polynomials in their univariate form,
$\Hlam_n(x,y),$ where $n$ is the degree, are the bivariate exceptional
Hermites, and finally $\Rlam_m(x,y)$, where $m$ is the difference of
the degrees of the numerator and denominator, are rational functions
made by dividing the bivariate exceptional Hermite polynomials by a
denominator polynomial $\taulam(x,y)$.  The span of the former will be
denoted by $\Uslam\subset \cP$ while the span of the latter will be
denoted by $\Wslam= \taulam{}^{-1}\Uslam$.

The symbol $\lambda$ will denote a partition and $\Mlam$ the
corresponding Maya diagram.  For each $\lambda$, the there are certain
canonically defined sets $\Ilam$ and $\Jlam\subset \Z$ that serve as
the index sets of $\Hlam_n$ and $\Rlam_m$, respectively.  The symbols
$\cKlam, \Glam_q,\; q\in \Z$ denote finite sets of integers that
encode various combinatorial properties of $\lambda$.  The symbols
$\kaplam(m), \gamlam_q(m),\; q\in \Z$ denote the corresponding monic
polynomials whose roots are precisely these sets.

The symbol $\cC$ will denote the vector space of distributions
generated by 1-point functionals\footnote{Wilson in \cite{wilson}
  refers to these as 1-point conditions; our $\cC$ is Wilson's
  $\mathcal{C}$.}.  The symbol $\cC_\zeta$ refers to the subspace of
functionals with support at a particular $\zeta\in \C$.  The symbol
$\Wlam\in \Grad$ refers to a point in the adelic Grassmannian that is
canonically associated to a partition $\lambda$.  These are the points
whose $\tau$-function is a Schur polynomial; they are discussed by
Wilson in Section 10, Example 2 of \cite{wilson}.  The bispectral
involution will be denoted by $\beta\colon \Grad\to \Grad$.  As in
Wilson, the application of a functional to a function will be denoted
using angle brackets, as in $\la c, f \ra$.  We will also use angle
brackets to denote the inner product relative to which the Hermite
polynomials are orthogonal.  To avoid any possible confusion, we will
add a subscripted $H$ and write $\la \cdot, \cdot\ra_H$ in such cases.

\section{Partitions and Schur Functions }\label{sect:SF}
\subsection{Partitions and Maya Diagrams}
A \textit{partition} $\lambda$ is a decreasing, non-negative integer sequence
$\lambda_1 \geq \lambda_2 \geq \cdots $ such that
\[ |\lambda| := \sum_{i=1}^\infty \lambda_i< \infty.\] Implicit in
this definition is the assumption that $\lambda_i=0$ for $i$
sufficiently large.  The length of $\lambda$, which we will denote by
${\ell(\lambda)}$, is the number of non-zero elements of the
sequence.

A closely related concept is that of a \textit{Maya diagram}.  A Maya
diagram is a subset of $\Z$ that contains a finite number of positive
integers and excludes a finite number of negative integers.  For a
given $\lambda$, define the strictly decreasing sequence
\begin{equation}
  \label{eq:midef} m_i(\lambda) = \lambda_i - i,\quad i=1,2,\ldots.
\end{equation}
The set
\begin{equation}
  \label{eq:Mlam}
  \Mlam = \{ m_i(\lambda) \colon i = 1,2,\ldots \}
\end{equation}
is a Maya diagram because $m_{i+1}(\lambda) = m_i(\lambda)-1$ for
$i\geq {\ell(\lambda)}+1$.  Conversely, if $M\subset \Z$ is a Maya
diagram, then
\[M=\Mlam + l = \{ m_i(\lambda)+l \colon i=1,2,\ldots \}  \]
for some partition $\lambda$ and $l\in \Z$.

To a partition $\lambda$ and an integer $l\geq\ell(\lambda)$, define
the \textit{index set} of length $l$ associated to $\lambda$ to be
\begin{align}
  \label{eq:cKlamn}
  \cKlam_{l} &= \{ m_1(\lambda) + l,\ldots, m_l(\lambda) + l \}.
\end{align}
Since $l\geq {\ell(\lambda)}$, it follows that $m_l(\lambda) +l \geq 0$ and
$m_{l+1}(\lambda) +l < 0$. Hence $\cKlam_{l}$ consists precisely of
the non-negative elements of $\Mlam + l$. 
Observe that when $l=\ell(\lambda)$ then
$ m_{l}(\lambda)+{\ell(\lambda)}>0$ and
$m_j(\lambda)+\ell(\lambda)<0,\; j> \ell(\lambda)$, by definition.  In
this case, it is convenient to drop the subscript and write
\begin{equation}
  \label{eq:cKlamdef}
  \cKlam:= \cKlam_{{\ell(\lambda)}} = \{ m_1(\lambda) + \ell(\lambda),
  \ldots,  m_\ell(\lambda) + \ell(\lambda) \}.
\end{equation}
Thus, $\cKlam$ is the smallest index set, and also the only index set
consisting of strictly positive elements.  The correspondence
$\lambda \mapsto \cKlam$ is a bijection between the set of partitions
and the set of finite subsets of $\N$.



For a partition $\lambda$, let
\begin{equation}
  \label{eq:Jlamdef}
  \Jlam := \Z \setminus \Mlam
\end{equation}
denote the complement of the corresponding Maya diagram\footnote{Note
  that $-\Jlam$ is itself a Maya diagram.}.  An integer can be
``inserted'' into the partition $\lambda$ to produce a new partition
as follows.  For $m\in \Jlam$ let $m\triangleright\lambda$ denote the
partition
\begin{equation}
  \label{eq:lmlist}
  \lambda_1 -1, \ldots , \lambda_j -1, m+j, \lambda_{j+1},
  \lambda_{j+2},\ldots,
\end{equation}
where $j$ is the smallest natural number such that
$m+j \geq \lambda_{j+1}$. The sequence \eqref{eq:lmlist} is a
partition because $m\in\Jlam$
implies that either $j=0$ and $m>m_1(\lambda)$, or that
\[ \lambda_j-j = m_j(\lambda) > m> m_{j+1}(\lambda) =
  \lambda_{j+1}-j-1.\] Another way to understand the transformation
$\lambda\mapsto m\triangleright\lambda$ is to observe that it adds one
element to the corresponding index sets. To wit,
\begin{equation}
  \label{eq:Klambdam}
  \cK^{(m\triangleright\lambda)}_{l+1} = \cKlam_{l} \cup \{ m+l\},\quad
  l\geq \ell(\lambda).  
\end{equation}

\subsection{Schur Functions}
For every $k\in \N$, define the ordinary Bell polynomials
$B_k(t_1,\ldots, t_k) \in \C[t_1,\ldots, t_k]$ as the
coefficients of the power generating function
\begin{equation}
  \label{eq:Bgf} \psi_0(\bt;z):=\exp\left(\sum_{k=1}^\infty t_k z^k\right) =
  \sum_{k=0}^\infty B_k(t_1,\ldots, t_k) z^k,\quad 
  \bt=(t_1,t_2,\ldots).
\end{equation} Since the above generating function  can also
be written as
\[ \psi_0(\bt;z) = \sum_{j=0}^\infty
  \frac{1}{\mu!}\left(\sum_{k=0}^\infty t_k z^k\right)^\mu ,\] the
multinomial formula implies that
\begin{equation}
  \label{eq:Bkmu}
  \begin{aligned} B_k(t_1,\ldots, t_k) &= \sum_{ \Vert \mu\Vert=k} \frac{t^{\mu_1}_1}{\mu_1!}
\frac{t^{\mu_2}_2}{\mu_2!}  \cdots \frac{t^{\mu_{k}}_{k}}{\mu_{k}!},\qquad
\Vert \mu \Vert = \mu_1 + 2\mu_2 + \cdots + {k} \mu_{k}\\ &=
\frac{t_1^k}{k!} + \frac{t_1^{k-2}t_2}{(k-2)!} + \cdots + t_{k-1} t_1
+ t_k.
  \end{aligned}
\end{equation}


For any partition $\lambda$, define the Schur function
$\Slam(t_1,\ldots, t_N)\in \Q[t_1,\ldots, t_N],\; N =|\lambda|$ to
be the multivariate polynomial
 \begin{equation}
   \label{eq:Slamdef} \Slam =
   \det(B_{\lambda_i-i+j})_{i,j=1}^l
 \end{equation}
 where $l$ is any integer satisfying $l\geq\ell(\lambda)$ and  $B_k=0$ when $k<0$.
 Moreover, since
 \[ \partial_{t_i} B_j(t_1,\ldots, t_j) = B_{j-i}(t_1,\ldots,
   t_{j-i}),\quad j\geq i,\] we may re-express \eqref{eq:Slamdef} in
 terms of a Wronskian determinant:
\begin{equation}
  \label{eq:Slamwronsk}
  \Slam = \Wr[B_{k_l},\ldots, B_{k_1}],
\end{equation}
where $k_1>\cdots > k_l$ are the elements of the index set $\cKlam_{l}$
defined in \eqref{eq:cKlamn}.  
It is important to note that \eqref{eq:Slamdef} and
\eqref{eq:Slamwronsk} yield the same polynomial $S_{\lambda}$
regardless of the value of $l\geq \ell(\lambda)$ \footnote{This is
  because increasing the value of $n$ by one has the effect of
  increasing the size of the matrix, adding a new first row and
  column, but since there is necessarily a $1$ in the top-left corner
  and zeroes below it, the determinant is unchanged.}.

The Schur functions, $\Slam$, are closely related to the
representation theory of the symmetric group on $n$ elements.  Such
irreducible representations are labelled by partitions $\lambda$ such
that $|\lambda|=N$.  The conjugacy classes of the symmetric group
correspond to cycle types $c_\mu=(1^{\mu_1},2^{\mu_2},\ldots)$ where
\[\Vert \mu \Vert := \sum_j j\,\mu_j = N. \]
It is known \cite[Section I.7]{macdonald} that
\begin{equation}
  \label{eq:Slambdachi}  
   \Slam(t_1,\ldots, t_N) = \sum_{ \Vert \mu\Vert=N}
  \chilam(c_\mu)\frac{t^{\mu_1}_1}{\mu_1!}
  \frac{t^{\mu_2}_2}{\mu_2!}  \cdots
  \frac{t^{\mu_{\ell(\lambda)}}_{\ell(\lambda)}}{\mu_{\ell(\lambda)}!} 
\end{equation}
where $\chilam$ denotes the character of the representation
labelled by $\lambda$.  By the hook-length formula, the coefficient of
$t_1^N$ in $\Slam$, is equal to
\begin{equation}
  \label{eq:hooklength}
  \frac{\chilam(1^N)}{N!} = \frac{d_\lambda}{N!}=
  \frac{\prod_{i<j\leq n} (k_i-k_j)}{\prod_i
    k_i! },
\end{equation}
where $d_\lambda$ is the dimension of the irreducible representation
corresponding to $\lambda$, and where $k_1>\cdots > k_l$ are the
elements of the index set $\cKlam_{l}$ defined in \eqref{eq:cKlamn}.

\section{Sato Theory and the Adelic Grassmannian}
\label{sect:grad}
\subsection{The Adelic Grassmannian}

For $k\in \N$ and $\zeta \in \C$ we let $\bev{k}{\zeta}:\cP\to \cP$
denote the evaluation functional composed with the $k\supth$-order
derivative:
\begin{equation}
  \label{eq:ekzeta}
  \la \bev{k}{\zeta}, f\ra = f^{(k)}(\zeta),\quad f\in \cP.
\end{equation}
Let
\[ \cC_\zeta = \lspan\{ \bev{k}{\zeta} : k\in \Nz \} \] denote the
vector space of all 1-point functionals with support
at a fixed $\zeta\in \C$ and let
\[\cC = \bigoplus_{\zeta\in \C} \cC_\zeta \]
be the vector space spanned by 1-point functionals with arbitrary
support.  As the need arises, we proceed with the understanding that
functionals in $\cC$ are also allowed to act on rational and analytic
functions (with the appropriate domain safeguards).  In situations
where a functional acts on a multi-variable function, we adopt the
convention that $c(z)$ indicates that $c\in \cC$ acts on a function of
the variable $z$.


We will say that a subspace $C\subset \cC$ is \emph{homogeneous} if it
has a basis of one-point functionals.  Thus, $C\subset \cC$ is
homogeneous if and only if
\[ C = \bigoplus_{\zeta} (C\cap \cC_\zeta) .\]

Let $C\subset \cC$ be a finite-dimensional subspace of 
differential functionals, and let
\begin{equation}
  \label{eq:UCdef}
  \Ker C = \{ f\in \cP : \la \cc,f\ra = 0 \text{ for all
  }  \cc\in C \} 
\end{equation}
be the joint kernel of the elements of $C$.  It is easy to show that
$\dim C=\codim \Ker C$, where the latter denotes the codimension of
$\Ker C\subset \cP$.   Dually, if $U=\Ker C$ for some finite-dimensional
$C\subset \cC$, then
\[C=\Ann U= \{ \cc\in \cC \colon \la \cc, f \ra = 0 \text{ for all } f
  \in U \}.\]

Let $C \subset \cC$ be a homogeneous, finite-dimensional subspace of
functionals.   Define
 \begin{equation}
   \label{eq:CWdef}
  q_C(z):= \prod_{i=1}^n (z-\zeta_i),\qquad
  W_C:= q_C^{-1} \Ker C
\end{equation}
where $\cc_i\in \cC_{\zeta_i},\; i=1,\ldots, n,\; \zeta_i\in\C$ is a
choice of basis for $C$.  It is evident that $q_C$ and $W_C$ are
independent of the choice of basis.

\begin{definition}
  We define $\Grad$, the adelic Grassmannian \cite{wilson}, to be the
  set of all subspaces of the form $W_C$ where $C\subset \cC$ is
  homogeneous.  We will say that $C\subset \cC$ is \emph{reduced} if
  for all $\zeta\in \C$ we have $\bev0\zeta\notin C$.  Equivalently,
  $C$ is reduced if and only if the elements of $\Ker C$ do not have a
  shared root.
\end{definition}

\begin{prop}
  \label{prop:Creduced}
  For every $W\in \Grad$ there exists a unique reduced homogeneous
  $\tC\subset \cC$ such that $W= W_{\tC}$.
\end{prop}

\subsection{KP  Wave\ functions and Wilson's Bispectral Algebras}
\label{sect:bfunc}
Sato theory associates a wave function and a rational solution of the
KP hierarchy to each point in $\Grad$ as follows \cite{Sato} (see also
\cite{wilson,SW}).

Let $W_C\in \Grad,\; C\in \cC$ as per \eqref{eq:CWdef}, and let
$\cc_i\in \cC_{\zeta_i},\; i=1,\ldots, l$ be a basis of $C$.
Let
\[ \KC = \partial_{t_1}^l + \sum_{i=1}^l a_i(\bt)
  \partial_{t_1}^{l-i} \] denote the monic differential operator whose
action on an arbitrary function $f$ is
\[
  \KC f(\bt) = \frac{\Wr[\phi_1(\bt),\ldots,\phi_l(\bt),f(\bt)] }{
    \Wr[\phi_1(\bt),\ldots,\phi_l(\bt)]}.
\]
where
\[ \phi_i(\bt) :=\langle \cc_{i},\psi_0(\bt,z) \rangle \] with
$\psi_0(\bt,z)$ the generating function of the Bell polynomials
previously introduced in \eqref{eq:Bgf}.

The
\textit{dynamical wave function} associated to $W=W_C$ is defined to be
\begin{equation}
\label{eq:PsiWdef}
\psi_W(\bt;z)=
\frac{1}{q_C(z)} \KC \psi_0(\bt,z)=
\frac{\Wr_{t_1}[\phi_1(\bt),\ldots,\phi_l(\bt),\psi_0(\bt,z)]}{q_C(z)\tau_C(\bt)}
\end{equation}
where
\begin{equation}
  \label{eq:tauC}
  \tau_C(\bt) =   \Wr_{t_1}[\phi_1(\bt),\ldots,\phi_l(\bt)].
\end{equation}
The dynamical wave function can also derived from the $\tau$-function
using the so-called Miwa shift \cite[Equation (5.16)]{SW}:
\begin{align}
  \label{eq:miwashift}
  \psi_W(\bt,z) &= \frac{\phi_W(\bt,z)}{\tau_C(\bt)}
                  \psi_0(\bt;z)\\\intertext{where} \nonumber 
  \phi_W(\bt,z) &= \tau_C(t_1-z^{-1}, t_2-1/2 z^{-2},\ldots).
\end{align}

Even though the definition of $\psi_W$ depends on a choice of
functionals $C$, the correspondence $W\mapsto \psi_W$ is well-defined
as a consequence of the following; c.f., Proposition \ref{prop:Creduced}.
\begin{prop}
  \label{prop:ACC'}
  Let $C\subset \cC$ be a homogeneous subspace of functionals with
  $U=\Ker C$ the corresponding polynomial subspace.  Let $r\in \cP$ be
  a monic polynomial and let
  \begin{equation}
    \label{eq:C'def}
    C_r = \{ c\in \cC \colon c\circ r\in C \} = \Ann (r U).
  \end{equation}
  Then,
  $W_{C_r} = W_C$ and $\KCp = \KC \circ r(\partial_{t_1})$
\end{prop}
\begin{cor}
  \label{cor:Cindep}
  The definition \eqref{eq:PsiWdef} of the wave function $\psi_W$ is
  independent of the choice of $C$.
\end{cor}

The dynamical wave function is fully characterized by the following
properties.
\begin{prop}
  \label{prop:Psiform}
  For an $l$-dimensional homogeneous $C\subset\cC$ and
  $W=W_C\in \Grad$, the corresponding wave function has
  the form
  \begin{equation}
    \label{eq:Psiform}
    \psi_{W}(\bt,z) = \frac{1}{q_C(z)} \lp z^l+ \sum_{i=1}^l
    \phi_i(\bt) z^{l-i}\rp \psi_0(\bt,z),
  \end{equation}
  where the coefficients $\phi_i(\bt)$ are rational functions, and
  where
  \begin{equation}
    \label{eq:cqCPsiW}
    \la \cof{z}, q_C(z) \psi_W(\bt,z) \ra = 0,
  \end{equation}
  for all $\bt$ and $c\in C$.
  Moreover, if \eqref{eq:cqCPsiW} holds for some $c\in \cC$, then
  necessarily $c\in C$.
\end{prop}

The connection to the KP hierarchy takes the form of the following
observations \cite{SW}. The pseudo-differential operator
\[ \cL_W(\bt,z)=\KC(\bt,z)\circ {\partial_{t_1}}\circ
  \KC^{-1}(\bt,z)\] satisfies the nonlinear evolution equations
\eqref{eqn:KPhierarchy} of the KP hierarchy.  \hlight{The ring
\begin{equation}
  \label{eq:cRW}
  \cR_W:=\{p\in\cP \ :\ p W\subset W\}
\end{equation}
is called the stabilizer of $W$.  Dually, the stabilizer ring
may be characterized as
\begin{equation}
  \label{eq:cRW1}
  \cR_W:=\{p\in\cP \ :\ \cc\circ p \in C\text{ for all } \cc\in C\}.
\end{equation}}
It follows that for every $p\in \cR_W$ we have that 
\begin{equation}
  \label{eq:Lpdef}
  L_p := p(\cL_W) = \KC(\bt,z) \circ p(\partial_{t_1}) \circ \KC^{-1}(\bt,z)
\end{equation}
is a differential operator.  Moreover, by construction, $L_p$ satisfies the
eigenvalue equation
\begin{equation}
  \label{eq:LpPsiW}
  L_p(\bt,z)\psi_W(\bt,z)=p(z)\psi_W(\bt,z).
\end{equation}
Since any polynomial
with a factor of $(q_C(z))^N$ is in $\cR_W(z)$ for sufficiently high powers of
$N$, this construction produces an algebra of differential operators
that is non-empty and includes every sufficiently high order.

Although the construction above was initially created to study the
dynamics of the KP hierarchy, the seminal paper by Wilson
\cite{wilson} used it to address the bispectral problem in the
following elegant way.  Rename the first time variable by setting
$x=t_1$ and ``turn off'' all of the other time variable by setting
$t_i=0$ for $i>1$.  Then the \textit{stationary wave function}
\begin{equation}
  \label{eq:psixzdef}
\psi_W(x,z)=\psi_W(x,0,0,0,\ldots;z)
\end{equation}
is an eigenfunction for a ring of ordinary differential operators in
$x$ with eigenvalues depending polynomially on $z$; this follows by
\eqref{eq:LpPsiW}.  Thus, \cite[Proposition 5.1]{wilson}
\begin{equation}
  \label{eq:psiai}
  \psi_W(x,z) = q_C(z)^{-1}\lp z^l + \sum_{i=1}^l \phi_i(x) z^{l-i}\rp
  \e^{xz},
\end{equation}
where $\phi_i(x),\; i=1,\ldots, n$ are rational functions uniquely
determined by the conditions
\[ \la c_i(z), (z^l + \sum_{i=1}^l \phi_i(x) z^{l-i}) \e^{xz} \ra =
  0,\quad i=1,\ldots, l,\] where $c_1,\ldots, c_l$ are a basis of $C$.
Wilson also showed \cite[Theorem 2]{wilson} that the relation
\begin{equation}
  \label{eq:psibetaW}
  \psi_W(z,x)=\psi_{\beta(W)}(x,z)
\end{equation}
defines an involution $W\mapsto \beta(W)$ on $\Grad$.
It follows that $\psi_W(x,z)$ is part of a bispectral triple in that
it is also the eigenfunction for differential operators in $z$ with
eigenvalues depending polynomially on $x$.

\section{$\Grad$ and Exceptional Hermite Polynomials}
\label{sect:XOP}
\subsection{Classical Hermite Polynomials}

Classical Hermite polynomials are orthogonal polynomials defined by
the recurrence relation
\begin{equation}
  \label{eq:herm3trr}
  h_0=1,\quad x h_n(x) = \frac12 h_{n+1}(x) + n h_{n-1}(x),\quad n=1,2,\ldots
\end{equation}
They are orthogonal with respect to the following inner product:
\begin{equation}
  \label{eq:hortho}
  \int_{\R} h_m(x) h_n(x) \e^{-x^2} dx = \sqrt{\pi}\, 2^n
  n! \delta_{n,m}
\end{equation}
and satisfy the following second-order eigenvalue equation
\begin{equation}
  \label{eq:hermde}
  h_n '' - 2x h_n' = - 2n h_n,\quad n=0,1,\ldots
\end{equation}
The Hermite polynomials may also be defined in terms of the Rodrigues
formula
\begin{equation}
  \label{eq:hermrr}
  h_n(x) = (-1)^n    \e^{x^2} \partial_x^n \e^{-x^2},\; n=0,1,2,\ldots
\end{equation}
Relation \eqref{eq:hermrr} entails the following representation of the
Hermite polynomials in terms of an exponential generating function
\begin{align}
  \label{eq:chermgf}
  \sum_{n=0}^\infty \frac{h_n(x)}{n!}\lp \frac{-z}{2}\rp^n
  &=    \e^{x^2} \exp\lp -\frac{z}{2}\, \partial_x \rp \e^{-x^2}
  = \e^{x^2-(x-z/2)^2}=  \e^{xz-\frac14 z^2}.  
\end{align}

Let us introduce a bivariate version of the Hermite polynomials, defined as
\begin{equation}
  \label{eq:Hxydef}
  H_n(x,y) := (- y)^{n/2} h_n\lp \frac{ x}{\sqrt{-4y}}\rp.
\end{equation}
The univariate Hermite polynomials can be recovered as
\[ h_n(x) = 2^{n} H_n(x,-1/4).\] The generating function for the
bivariate polynomials takes the form:
\begin{equation}
  \label{eq:Psigf}
\Psi_0(x,y,z):= \exp(x z + y z^2)=
  \sum_{n=0}^\infty H_n(x,y) \frac{z^n}{n!}.
\end{equation}
It follows from relation \eqref{eq:Psigf} that $H_n(x,y)$ is monic in
$x$ and is weighted degree homogeneous relative to the grading
\begin{equation}
  \label{eq:degx1degy2}
 \deg x =1,\; \deg y = 2.
\end{equation}



A number of fundamental identities involving Hermite polynomials
follow from \eqref{eq:Psigf}.  For example, mirroring the argument of
\eqref{eq:chermgf}, the bivariate version of the Rodrigues formula
takes the form
\begin{equation}
  \label{eq:h2rr}
  H_n(2xy,y) =  \e^{-y x^2} \partial_x^n\e^{y x^2},\quad n=0,1,2,\ldots.
\end{equation}
By inspection, $\Psi_0(x,y,z)$ is annihilated
by the operator $2y \partial_{x}^2 + x\partial_x - z\partial_z$.  Since
\[ z\partial_z \Psi_0(x,y,z) = \sum_{n=0}^\infty n H_n \frac{z^n}{n!} \]
this observation entails the following, scaled version of the Hermite
differential  equation:
\begin{equation}
  \label{eq:hermde2}
  T(x,y,\partial_x) H_n(x,y) = n H_n(x,y),\quad\text{ where }\quad
  T(x,y,z)= 2yz^2 + xz.
\end{equation}
Applying the scaling transformation \eqref{eq:Hxydef} to the classical
orthogonality relation \eqref{eq:hortho} yields the following
scaled orthogonality relation.  For fixed $y<0$, we have
\begin{equation}
  \label{eq:intHnHm}
  \la H_{n_1}(x,y), H_{n_2}(x,y)\ra_H = \nu_n(y)\,
  \delta_{n_1,n_2},\quad n\in \Nz,
\end{equation}
where
\begin{equation}
  \label{eq:expx2ip}
  \la f(x),g(x)\ra_H = \int_\R f(x) g(x) \e^{\frac{x^2}{4y}}dx,\quad y<0
\end{equation}
and where
\begin{equation}
  \label{eq:nunydef}
  \nu_n(y) = 2(-\pi y)^{1/2}\,(-2y)^{n}\, n!. 
\end{equation}

Specializing the
generating function for Bell polynomials \eqref{eq:Bgf}, gives the
following, well-known, representation of Hermite polynomials as a
finite sum:
\begin{equation}
  \label{eq:HBell}
  H_n(x,y) = n! B_n(x,y,0,\ldots) = \sum_{j=0}^{\lfloor n/2\rfloor} \frac{n!}{(n-2j)!
    j!} x^{n-2j} y^j
\end{equation}

Next, consider the 1st order eigenvalue relation:
\[ \partial_x \Psi_0(x,y,z) = z \Psi_0(x,y,z) \] This relation entails
the well-known lowering identity
\begin{equation}
  \label{eq:DxHn}
  \partial_x H_n(x,y) = n H_{n-1}(x,y).
\end{equation}
In more combinatorial language, we can say that $H_n(x,y),\; n\in \Nz$
forms an Appell sequence \cite{RKO}.

Similarly, the relation
\begin{equation}
  \label{eq:3termdiffrel}
  (\partial_z-2yz) \Psi_0(x,y,z)  = x\Psi_0(x,y,z) .
\end{equation}
entails the bivariate version of the recurrence relation
\eqref{eq:herm3trr}, namely:
\begin{equation}
  \label{eq:class3term}
  \Theta_1(n,y,\sS_n) H_n(x,y) = x H_n(x,y),\quad\text{ where } \quad
 \Theta_1(n,y,z)  =z-2ynz^{-1}
\end{equation}
where $\sS$ is the unit right-shift operator.

\subsection{Exceptional polynomials}
\label{sect:xop}
Exceptional Hermite polynomials \cite{GGM} are a far ranging
generalization of the classical Hermite polynomials.  Just like their
classical counterparts, exceptional polynomials satisfy a second-order
eigenvalue equation. The key difference is that the resulting
polynomial family has a finite number of missing, exceptional degrees.

Let $\lambda$ be a fixed partition and set $N=|\lambda|,\; \ell =
\ell(\lambda)$. 
Let $k_1>\cdots > k_{\ell}$ be the elements of the corresponding index
set $\cKlam$ as per \eqref{eq:cKlamdef}.  In the existing literature,
exceptional Hermite polynomials associated to the partition $\lambda$
are defined as the Wronskian of classical Hermite polynomials,
\[ h^{(\lambda)}_{k+N-\ell} = \Wr[h_{k_{\ell}},\ldots , h_{k_1},
  h_{k}],\quad k\notin \cKlam.\]


As was the case with classical Hermite
polynomials, we define a bivariate version of exceptional Hermite
polynomials.  Observe that the map
\begin{equation}
  \label{eq:HkWr}
  H_k\mapsto\Wr[H_{k_\ell},\ldots,H_{k_1},H_k],\quad k\notin
  \cKlam, 
\end{equation}
changes the degree by $N-\ell$.
Set
\begin{equation}
  \label{eq:Ilamdef}
  \Ilam = \Jlam+N,
\end{equation}
and observe that if $n\in \Jlam$, then $k = n-N+\ell\notin \cKlam$ is
a valid index for the Wronskian in \eqref{eq:HkWr}.  We are thus able
to define a non-zero polynomial
\hlight{\begin{equation}
  \label{eq:Hlamdef}
  \Hlam_{n} :=    \frac{\Wr[H_{k_\ell},\ldots,
    H_{k_1},H_{n-N+\ell}]}{ 
    \prod_{i<j}(k_i-k_j)\prod_{i}
    (n-N+\ell - k_i)},\quad n\in \Ilam
\end{equation}}
Observe that the ``degree shift'' of the index in \eqref{eq:Hlamdef}
ensures precisely that the exceptional polynomial $\Hlam_n(x,y)$ has
degree $n$ in $x$.  Furthermore, $\Hlam_n(x,y)$ is
weighted-homogeneous relative to \eqref{eq:degx1degy2} and monic in
$x$.
\begin{prop}
    The polynomial family $\{ H^{(\lambda)}_{n} \colon n\in \Ilam \}$
    is missing the \emph{exceptional} degrees
    \begin{equation}
      \label{eq:lamexcept}
      \cKlam_{N} = \{ 0,1,\ldots, N-\ell-1 \} \cup \{
      \lambda_\ell+N-\ell ,\ldots,       \lambda_1 +N-1 \}.  
    \end{equation}
  \end{prop}
  \begin{proof}
    By the remark just after \eqref{eq:Slamwronsk}, we have
    \begin{equation}
      \label{eq:HlamN}
      \Hlam_{n} :=    \frac{\Wr[H_{k_N},\ldots,
        H_{k_1},H_{n}]}{ 
    \prod_{i<j}(k_i-k_j)\prod_{i}
    (n - k_i)},\quad n\in \Ilam
\end{equation}
where $k_1,\ldots, k_N$ is an enumeration of $\Klam_N$.  By
\eqref{eq:cKlamn} and the remark that follows, $\cKlam_N$ consists of
non-negative elements of $\Mlam+N$.  Conclusion \eqref{eq:lamexcept}
follows because
\[ \lambda_1-1>\lambda_2-2, \cdots > \lambda_\ell -\ell > -\ell-1 > \ell-2 > \cdots > -N,
  \cdots \]
is a decreasing enumeration of $\Mlam$.
\end{proof}

Analogously to \eqref{eq:Hxydef}, the bivariate and univariate
Wronskians are related by
\begin{equation}
  \label{eq:xopscaling}
  \begin{aligned}
    &\Wr[H_{k_\ell},\ldots,H_{k_1},H_k](x,y) \\
    &\qquad = 2^{\frac12
      \ell(\ell+1)}(-y)^{(k+N-\ell)/2}\Wr[h_{k_\ell},\ldots , h_{k_1},
    h_{k}]\lp \frac{x}{\sqrt{-4y}}\rp
  \end{aligned}
\end{equation}
Thus, one could define $\Hlam_n(x,y)$ by appropriately scaling and
normalizing the univariate Wronskian
$\Wr[h_{k_\ell},\ldots , h_{k_1}, h_{n-\delta}]$, but
\eqref{eq:Hlamdef} is more direct.

For notational convenience, let
\begin{equation}
  \label{eq:taulamdef}
  \taulam(x,y) := \frac{N!}{d_\lambda}    \Slam(x,y,0,\ldots),
\end{equation}
with $d_\lambda$ as per the hook-length formula \eqref{eq:hooklength}.
Inspection of \eqref{eq:Slambdachi} shows that $\taulam(x,y)$ is a
monic polynomial in $x$ of degree $N$ and weighted-homogeneous
relative to \eqref{eq:degx1degy2}.  Hence, $\taulam$ is nothing other
than the Schur function $\Slam$ with all but the first two variables
set to zero, renormalized so as to be monic.  The notation was chosen
to hint at the connection to the $\tau$-functions of integrable
systems, but the main point here is the observation that the
exceptional Hermite polynomials associated to a partition can be
expressed simply in terms of the Schur functions produced from that
partition via insertion:
\begin{thm}
  \label{thm:tauschur}
  The exceptional Hermite polynomials are given by:
  \begin{equation}
    \label{eq:defHlamn}
    H^{(\lambda)}_{n}=  \tau_{(n-N)\triangleright
      \lambda},\quad n\in \Ilam.
  \end{equation}
\end{thm}
\begin{proof}
  By \eqref{eq:Slamwronsk}, \eqref{eq:hooklength} and
  \eqref{eq:HBell}, we have
  \begin{equation}
    \label{eq:taulamWr}
    \taulam=
    \frac{\Wr[H_{k_\ell},\ldots, H_{k_1}]}{    \prod_{i<j\leq n}(k_i-k_j)}.
  \end{equation}
  The desired conclusion now follows by \eqref{eq:Klambdam}.
\end{proof}

Before continuing let us also note the following generalization of
\eqref{eq:HBell}.  
\begin{cor}
  Let $\chilam$ be the character of the $\lambda$-irrep of the
  symmetric group on $N$ objects, and $c_j:=(2^j, 1^{N-2j})$ the
  indicated cycle type.  Then,
  \begin{equation}
    \label{eq:WRHchar}
    \Wr[H_{k_\ell},\ldots, H_{k_1}](x,y)= \prod_{i<j\leq
      n}(k_i-k_j)\sum_{i=0}^{\lfloor N/2 \rfloor} \chilam(c_j)
    \frac{N!}{(N-2j)!  j!}\,x^{N-2j} y^j 
  \end{equation}
\end{cor}
\noindent
This result was first announced in \cite{bondunstev}, where it was
proved using a different method.

Next, define the differential operator
\begin{equation}
  \label{eq:Tlamdef}
  \Tlam(x,y,\partial_x) = 2y\partial_{x}^2 + \left( x -
    4y\,\frac{\taulam_x(x,y)}{\taulam(x,y)}\right)
  \partial_x + 
  \left( 2y\,\frac{\taulam_{xx}(x,y)}{\taulam(x,y)} -x \,
    \frac{\taulam_x(x,y)}{\taulam(x,y)}\right).
\end{equation}
The above expression is called an exceptional operator because it
admits polynomial eigenfunctions for all but the finite number of
exceptional degrees in \eqref{eq:lamexcept}.

\begin{prop}\label{prop:Tlam}
  The exceptional Hermite polynomials, $\Hlam_n,\; n\in \Ilam$
  are eigenfunction of $\Tlam$ with
  \begin{equation}
    \label{eq:THn}
    \Tlam \Hlam_n =  (n-N) \Hlam_{n}.
  \end{equation}
\end{prop}
\noindent
We postpone the proof until Proposition \ref{prop:TRmR}, below. Note
that, since $\tau^{(\emptyset)}=1$, the classical Hermite differential
equation \eqref{eq:hermde2} is the particular case of the above result
corresponding to the trivial partition.

Modulo certain regularity assumptions, the polynomials
$\Hlam_n(x,y),\; n\in \Ilam$ constitute an orthogonal family.
Say that $\lambda$ is an \emph{even} partition if ${\ell}$ is even and
if $\lambda_{2i-1} = \lambda_{2i}$ for every $i=1,\ldots, {\ell}/2$.
Equivalently, $\lambda$ is even if and only if
$\kaplam(m) \geq 0$ for all $m\in\Jlam$.  The following
result was proved by Krein and Adler (see \cite{GGM}).
\begin{prop}
  If $\lambda$ is an even partition and $y<0$ is fixed, then
  $\taulam_y(x):= \taulam(x,y)$ has no real zeros.
\end{prop}
\noindent
Moreover, we have the following, proved in \cite{GGM}.  Set
\begin{align}
  \wlam_y(x)
  &=  \frac{\e^{\frac{x^2}{4y}}}{\taulam_y(x)^2},\\ 
  \label{eq:nulamdef}
  \nulam_m(y)
  &= 
    2(-\pi y)^{1/2} (-2y)^m \frac{(m+\ell)!}{\kaplam(m)}=2^{-\ell}
    \frac{\nu_{m+\ell}(y)}{\kaplam(m)},\quad 
    m\in \Jlam.
\end{align}
where
\begin{equation}
  \label{eq:kaplamdef}
  \kaplam(m) = \prod_{i=1}^{\ell} (m-m_i(\lambda)).
\end{equation}

\begin{prop}
  \label{prop:Hlamortho}
  If $\lambda$ is an even partition and $y<0$ is fixed, then the
  corresponding sequence of polynomials $\Hlam_m(x,y),\; n\in\Jlam$ is
  complete and orthogonal relative to $\wlam_y(x)dx$.  Indeed, for
  $n_1,n_2\in \Ilam$, we have
  \begin{equation}
    \label{eq:xoporth}
    \int_{-\infty}^\infty \Hlam_{n_1}(x,y) \Hlam_{n_2}(x,y)
    \wlam_y(x)dx =  \nulam_{n-N}(y) 
    \delta_{n_1n_2}.
  \end{equation}
\end{prop}


It turns out that the eigenvalue relation \eqref{eq:THn} and the
orthogonality relation \eqref{eq:xoporth} are easier to express and
understand if we change gauge and consider the following
\emph{exceptional rational functions}.    Let
\begin{equation}
  \label{eq:kaplamndef}
  \kaplamn{l}(m) =  \prod_{k\in  \cKlam_{l}} (m-k) =
                   \kaplam(m-l)F_{l-\ell}(m),\quad l\geq \ell,
\end{equation}
denote the monic polynomial with simple zeroes precisely at the
elements of $\cKlam_l$ as defined in \eqref{eq:cKlamn}. We can now
define
\begin{align}
  \label{eq:Rlamm}
    \Rlam_m
    &=\frac{\tau^{(m\triangleright
      \lambda)}}{\taulam}  = \frac{\Hlam_{m+N}}{\taulam} ,\quad m\in
      \Jlam.\\
      \intertext{Equivalently, for $l\ge \ell$, we have}
    \Rlam_{k-l}  &=\frac{1}{\kaplam_l(k)}
    \frac{\Wr[H_{k_{l}},\ldots, H_{k_1},
      H_{k}]}{\Wr[H_{k_{l}},\ldots, H_{k_1}]}   ,\quad 
    k\notin\cKlam_l.
\end{align}

The eigenvalue and orthogonality relations are now easier to
formulate. Set
\begin{equation}
  \label{eq:tTdef}
  \begin{aligned}
    \tTlam(x,y,\partial_x)
    &= (\taulam)^{-1} \circ \Tlam \circ \taulam\\
    &= 2y\partial_x^2+ x\partial_x + 4 y \lp\log  \taulam(x,y)\rp_{xx} 
  \end{aligned}
\end{equation}
\begin{prop}
  \label{prop:TRmR}
 With the above definitions, we have
 \begin{equation}
   \label{eq:TRmR}
   \tTlam\Rlam_m = m \Rlam_m,\quad m\in \Jlam.
 \end{equation}
\end{prop}

\noindent
This result was proved in \cite{GGM} and \cite{GKKM}, but we will give
a novel, simplified proof in Section \ref{sect:intertwiner} once we
introduce the intertwining operator.

Note that \eqref{eq:xoporth} may be restated quite simply as
\begin{equation}
  \label{eq:Rlamorth}
  \la \Rlam_{m_1}(x,y) ,\Rlam_{m_2}(x,y) \ra_H = \delta_{m_1,m_2}
  \nulam_{m_1}(y),\quad m_1, m_2\in \Jlam,
\end{equation}
where the inner product is
the same as in \eqref{eq:expx2ip}.  The orthogonality of
$\Rlam_m,\; m\in \Jlam$ stems from the fact $\tTlam$ is a symmetric
operator relative to the above inner product.  This, in turn, is a
consequence of the fact that the classical
$T(x,y,\partial_x)$ is symmetric
relative to the same inner product, and the fact that $\tT$ is a
modification of $T$ by a zeroth order term.

\subsection{Semi-Stationary Wave Functions as Generating Functions}
Thus far we have considered dynamical wave functions depending on the
infinitely many variables of the KP hierarchy and {stationary wave
  functions} obtained from them by setting all time variables except
the first equal to zero.  It turns out that exceptional Hermite
polynomials are best studied in the intermediate case in which the
first \textit{and} second KP time variables are retained.

Note, for example, that the generating function \eqref{eq:Psigf} for
the bivariate form of the classical Hermite polynomials is a
restricted vacuum wave\ function in which all time variables $t_i$ for
$i>2$ have been set to zero:
\begin{equation}
  \psi_0(x,y,0,0,\ldots,z) =
  \exp(xz + y z^2)=\Psi_0(x,y,z).
\end{equation}

The main result of this section is to demonstrate the exceptional
Hermite polynomials are similarly generated by the wave functions of
certain points in $\Grad$ indexed by partitions.  Many of their known
properties and answers to some open questions concerning them can be
derived from the bispectrality of these generating functions and
Wilson's bispectral involution.  We will return to this point in the
sections to follow.


Fix a partition $\lambda$, and let
$N=|\lambda|,\ell=\ell(\lambda)$. Define $\Wlam \in \Grad$ as
\begin{equation}
  \label{eq:Wlamdef}
  \Wlam(z):= \lspan \{  z^m \colon m\in\Jlam \},
\end{equation}
where $\Jlam$ is the complement of the corresponding Maya diagram as
per \eqref{eq:Jlamdef}.  Set
\begin{equation}
  \label{eq:Clamdef}
  \Clam    =\lspan\left\{ \bev{k}0\colon
    k\in \cKlam\right\}
\end{equation}
\begin{prop}
  \label{prop:WClam}
  We have $\Wlam = W_{\Clam}$. 
\end{prop}
\begin{proof}
  By construction,
  $\ker \Clam(z) = \{ z^k \colon k\in \Nz\setminus\cKlam\}$.  By
  \eqref{eq:Jlamdef} and \eqref{eq:Wlamdef}, it follows that
  \[ W_{\Clam}(z) = z^{-\ell}  \ker \Clam(z) = \Wlam(z) .\]
\end{proof}

\begin{definition}
  We refer to
  \begin{equation}\label{eqn:semiwave}
    \Psilam(x,y,z)
    := \psi_{\Wlam}(x,y,0,\ldots; z),
\end{equation}
obtained by letting $x=t_1$, $y=t_2$ and $t_i=0$ for $i>2$ in the
dynamical wave functions the \textit{semi-stationary wave function}
associated to $\Wlam$.
\end{definition}
Just as relation \eqref{eq:Psigf} shows that the vacuum wave function
serves as a generating function for the classical Hermite polynomials,
the semi-stationary wave function $\Psilam(x,y,z)$ serves as a
generating function for the corresponding exceptional Hermite rational
functions.  To be more precise, we have the following.
\begin{thm}\label{thm:generatingxops}
  We have
  \begin{align}
    \label{eq:Psilamgf}
    \Psilam(x,y,z)
    &= \sum_{m\in\Jlam}
      \frac{\kaplam(m)}{(m+\ell)!}\Rlam_m(x,y)\, z^m  ,\\
        \label{eq:Psilamgf0}
    &= \sum_{m=-\ell}^\infty
      \frac{\kaplam(m)}{(m+\ell)!}\Rlam_m(x,y)\, z^m ,\\
    \label{eq:Psilamgf1}
    z^N\taulam(x,y)\Psilam(x,y,z)
    &= \sum_{n\in \Ilam}  \frac{\kaplamn{N}(n)}{n!}\;\Hlam_n(x,y)\, z^n,\\
    \label{eq:Psilamgf10}
    &= \sum_{n=0}^\infty  \frac{\kaplamn{N}(n)}{n!}\;\Hlam_n(x,y)\,z^n,
  \end{align}
  with $\kaplam(m), \kaplamn{N}(n)$ the polynomials defined in
  \eqref{eq:kaplamdef} and \eqref{eq:kaplamndef}.
\end{thm}
\noindent Note that $\kaplam(m)=0$ precisely for those $m\ge -\ell$
for which $m\notin \Jlam$.  Thus \eqref{eq:Psilamgf0} is sensible
despite the fact that $R_m$ is not defined when $m\notin \Jlam$.  A
similar remark applies to \eqref{eq:Psilamgf10}.
\begin{proof}
  Let $k_1>\cdots > k_\ell$ be the elements of $\cKlam$ as per
  \eqref{eq:cKlamdef}.  By \eqref{eq:Clamdef},
  $\bev{k_i}0,\; i=1,\ldots, \ell$ is a basis for the annihilator of
  $z^\ell\Wlam(z)$. By \eqref{eq:Psigf},
  \[ H_k(x,y) = \la \bev{k}0(z), \Psi_0(x,y,z)\ra,\quad k\in \Nz.\]
  Hence, by
  \eqref{eq:PsiWdef}
  \begin{equation}
    \label{eq:PsilamWr}
    \begin{aligned}
      \Psilam(x,y,z) &= \frac{\Wr_x[H_{k_\ell}(x,y),\ldots,
        H_{k_1}(x,y), \Psi_0(x,y, z)]}{ \Wr_x[H_{k_\ell}(x,y),\ldots,
        H_{k_{1}}(x,y)] z^{\ell}}\\
      &=\sum_{n=0}^\infty \frac{\Wr_x[H_{k_\ell}(x,y),\ldots,
        H_{k_1}(x,y), H_n(x,y)]}{ \Wr_x[H_{k_\ell}(x,y),\ldots,
        H_{k_1}(x,y)] } \frac{z^{n-\ell}}{n!}
    \end{aligned}
  \end{equation}
  By \eqref{eq:Rlamm}, for $m\in \Ilam$ and $n=m+N$, we have
  \[
    \begin{aligned}
      \frac{\Wr_x[H_{k_\ell}(x,y),\ldots, H_{k_1}(x,y), H_n(x,y)]}{
        \Wr_x[H_{k_1}(x,y),\ldots, H_{k_{\ell}}(x,y)]} &
      =  \kaplam(m) \Rlam_{m}(x,y),
    \end{aligned}
  \]
  which entails \eqref{eq:Psilamgf}.  Relation \eqref{eq:Psilamgf1}
  follows by \eqref{eq:defHlamn}, \eqref{eq:Rlamm} and
  \eqref{eq:kaplamndef}.
\end{proof}

By \eqref{eq:miwashift} and \eqref{eq:taulamdef}, the semi-stationary
wave function can also be given as
\begin{equation}
  \label{eq:Psimiwa}
  \Psilam(x,y,z)
  =  \frac{\Philam(x,y,z)}{\taulam(x,y)}\Psi_0(x,y,z) 
\end{equation}
where
\begin{equation}
  \label{eq:Philamdef}
  \Philam(x,y,z) =  \frac{N!}{d_\lambda}
  \Slam\lp x-z^{-1}, y-2^{-1}z^{-2}, -3^{-1}z^{-3},\ldots, -
  N^{-1}z^{-N}\rp ,
\end{equation}
and where $\Slam$ is the Schur function defined in
\eqref{eq:Slamwronsk}. By \eqref{eq:Slambdachi},
$\Slam(t_1,\ldots, t_N)$ is weighted-homogeneous of degree $N$
relative to the grading $\deg t_i = i$. It follows that
$\Philam(x,y,z)$ is weighted-homogeneous of degree $N$, relative to
the grading
\begin{equation}
  \label{eq:xyzgrading}
  \deg x = 1,\quad \deg y = 2,\quad \deg z = -1.
\end{equation}


Let $\Uslam$ denote the $\cP$-module spanned by exceptional Hermite
polynomials:
\begin{equation}
  \label{eq:Uslamdef}
  \Uslam(x,y) = \lspan\{\Hlam_n(x,y)\colon n\in \Ilam\} \otimes \C[y],
\end{equation}
and let
\begin{equation}
  \label{eq:Wslamdef}
  \Wslam = (\taulam)^{-1} \Uslam.
\end{equation}
We will derive a number of results regarding exceptional Hermite
polynomials by manipulating meromorphic generating functions that have
a Laurent series expansion of the form
\begin{equation}
  \label{eq:Psiseries}
  \Psi(x,y,z) = \sum_{m\in\Jlam}  F_m(x,y) z^m,\quad F_m \in \Wslam.
\end{equation}
\begin{definition}
  For a given partition $\lambda$, we will call
  $\Phi(x,y,z) \in \C[x,y,z,z^{-1}]$ a $\lambda$-generator if
  \[ \Psi(x,y,z) =\frac{\Phi(x,y,z)}{\taulam(x,y)} \e^{x z+ yz^2} \]
  has the form shown in \eqref{eq:Psiseries}.  We will use $\Flam$ to
  denote the set of all $\lambda$-generators.
\end{definition}

The semi-stationary wave function is the canonical example of a
$\lambda$-generator with $\Phi = \Philam$.  Also, observe that
multiplication by a polynomial in $y$ preserves
\eqref{eq:Psiseries}.  For this reason we regard $\Flam$ as a
$\cP$-module rather than a vector space.
In section \ref{sect:1ptcond}, we will characterize $\Flam$ in term
of 1-point functionals.


\subsection{The intertwiner}
\label{sect:intertwiner}
Let $\lambda$ be a partition.  Let $N=|\lambda|, \ell=\ell(\lambda)$,
and let $\Philam$ be as in \eqref{eq:Philamdef}. Set
\begin{equation}
  \label{eq:Kslamdef}
  \Kslam(x,y,z) =  \frac{z^\ell\Philam(x,y,z)}{\taulam(x,y)}
  = z^{\ell} + \sum_{i=1}^{\ell}
  \frac{\Ks_i(x,y)}{\taulam(x,y)} z^{{\ell}-i} , 
\end{equation}
and observe that by \eqref{eq:PsilamWr}, the coefficients
$\Ks_i(x,y)\in \C[x,y]$ are weighted-homogeneous of degree
$i=1,\ldots,\ell$.
Recalling the convention set forth in Section~\ref{sec:conventions} regarding the
substitution of elementary derivative operators into multi-variable
polynomials we then have the operator
\[ \Kslam(x,y,\partial_{x}) := \partial_{x}^{\ell} + \sum_{i=1}^{\ell}
  \frac{\Ks_i(x,y)}{\taulam(x,y)} \partial_{x}^{{\ell}-i},\] which we
refer to as the semi-stationary intertwining operator. The choice of
terminology is justified by the following.
\begin{prop}
  \label{prop:KslamPsi}
  We have,
  \begin{equation}
  \label{eq:KslamPsi}
  \begin{aligned}
    \Kslam(x,y,\partial_{x}) \Psi_0(x,y,z)
    &= z^{\ell} \Psilam(x,y,z).
  \end{aligned}
\end{equation}
\end{prop}
\begin{proof}
  It suffices to  observe that $\partial_x \Psi_0(x,y,z)  = z\Psi_0(x,y,z)$.  
\end{proof}
\noindent By \eqref{eq:PsilamWr}, an equivalent definition of the
intertwiner is
\begin{equation}
  \label{eq:AlamWr}
  \Kslam f = 
  \frac{\Wr[H_{k_\ell},\ldots, H_{k_1},
    f]}{
    \Wr[H_{k_\ell},\ldots, H_{k_{1}}] }
\end{equation}
where $k_1,\ldots, k_\ell$ enumerate the index set $\Klam$.

By \eqref{eq:Slambdachi},  we may write
\begin{equation}
  \label{eq:Philamx}
  \Philam(x,y,z)= \sum_{i=0}^N \Phi_i(y,z) x^{N-i},
\end{equation}
where the coefficients
$\Phi_i(y,z)\in \C[y,z,z^{-1}]$ are weighted-homogeneous of degree
$i$.
Also note that $\Phi_{0}(y,z) = 1$ as a consequence of
\eqref{eq:hooklength}; that is $\Philam(x,y,z)$ is monic in
$x$.


\begin{lem}
  \label{lem:yz2}
  Let $\bH(y,z)$ denote the umbral operator \cite{RKO} whose action
  on a polynomial $\phi(x) = \sum_i \phi_i x^i$ is
  \begin{equation}
    \label{eq:bHdef}
    \bH(y,z)\phi(x) := \sum_i \phi_i H_i(x+2yz,y).
  \end{equation}
  Then, $\bH(y,z)$ is a 1-parameter transformation group with
  respect to $y$; that is,
  \begin{equation}
    \label{eq:bHgroup}
    \bH(y_1+y_2,z) = \bH(y_1,z) \circ \bH(y_2,z),\quad \bH(y,z)^{-1} =\bH(-y,z).
  \end{equation}
  Moreover, for $\pi\in \C[x,y,z]$ and
  $\hpi(x,y,z)=\bH(y,z)\pi(x,y,z)$, we have
  \begin{equation}
    \label{eq:baconj} \pi(\partial_z,y,z)\e^{xz+yz^2} = \hpi(x,y,z) \e^{xz+yz^2}.
  \end{equation}
\end{lem}
\begin{proof}
  By \eqref{eq:DxHn}, $H_n(x,y)$ is a Appell sequence.  Hence,
  \[
    H_{j}(x+2yz,y)= \sum_{k=0}^j\binom{j}{k} H_{j-k}(2 y z,y) x^k,\]
  By \eqref{eq:h2rr},  we have
  \[\e^{-yz^2} \circ \partial_z^j \circ \e^{yz^2}
    = \sum_{k=0}^j \binom{j}{k} H_{j-k}(2 y z,y) \partial_z^k .\]
  Hence,
  \[ \hpi(\partial_y,y,z)= \e^{-yz^2}\circ \pi(\partial_z,y,z)\circ
    \e^{yz^2}.\] Relation \eqref{eq:bHgroup} follows.  Moreover,
  \[
    \pi(\partial_z,y,z) \e^{xz+yz^2}
     = \e^{yz^2} \hpi(\partial_z,y,z)\e^{xz} = \hpi(x,y,z)\e^{xz+yz^2}.
   \]
\end{proof}

Using \eqref{eq:bHdef} we now define the dual intertwining operator,
\begin{equation}
  \label{eq:Klamdef}
  \begin{aligned}
    \Klam(x,y,z)
    &= \bH(-y,z) \Philam(x,y,z)\\
    &= H_N(x-2yz,-y) + \sum_{i=1}^N \Phi_i(y,z) H_{N-i}(x-2 y z,-y)
  \end{aligned}
\end{equation}

\begin{prop}
  \label{prop:KlamPsi}
  The dual intertwiner $\Klam(\partial_z,y,z)$ is a monic differential
  operator of order $N$.  Moreover,
  \begin{equation}
   \label{eq:KlamPsi} \Klam(\partial_z,y,z) \Psi_0(x,y,z) =
\taulam(x,y) \Psilam(x,y,z).
 \end{equation}
\end{prop}

\begin{proof}
  The first assertion now follows directly from the definition
  \eqref{eq:Klamdef}.  By \eqref{eq:bHgroup}, we have
  $\Phi(x,y,z) = \bH(y,z)\Klam(x,y,z)$.  The second assertion now
  follows by \eqref{eq:baconj}.
\end{proof}

One useful consequence of \eqref{eq:KlamPsi}, is a formula giving the
exceptional Hermite polynomials as a linear combination of the
classical Hermite polynomials whose coefficients are obtained
straightforwardly from $\Klam$.
Let $\Knat(n,y)$ be the unique polynomial characterized by the relation
\begin{equation}
  \label{eq:AlamAnat}
  ( \Klam(z,yz^2,\partial_z)\circ \partial_z^N) z^n
  =\Knat(n,y) z^n.
\end{equation}
Note that $\partial^{k}_z z^n = F_{k}(n) z^{n-k}$ where
\begin{equation}
  \label{eq:ffacdef} F_k(x) = \frac{\Gamma(x+1)}{\Gamma(x-k+1)}=
  \begin{cases}
    1 & k = 0\\
    x(x-1)\cdots (x-k+1),& k=1,2,\ldots \\
    \lp(x+1)(x+2)\cdots (x+k)\rp^{-1} & k  = -1,-2,\ldots
  \end{cases}
\end{equation} denotes the generalized falling factorial.  Thus, 
\begin{equation}
  \label{eq:Anatdef}
  \Knat(n,y) = \sum_{j=0}^N \Knat_j(n) y^j = \sum_{j=0}^{N}
  \sum_{i=0}^{N-j} K_{ij}\, y^j  F_{i+2j}(n).
\end{equation}
where $K_{ij}$ are the coefficients of $K$ as per
\begin{equation}\label{eq:Acoeff}
  \Klam(x,y,z) =
  \sum_{j=0}^{N} \sum_{i=0}^{N-j} K_{ij}\, x^{i}y^j
  z^{i+2j-N}.
\end{equation}

\begin{thm}\label{thm:Hlamlincomb}
  The expression $\Knat(n,y)/\kaplamn{N}(n)$, where the denominator is
  the polynomial defined in \eqref{eq:kaplamndef}, is a monic
  $N\supth$-degree polynomial in $y$ whose coefficients are
  polynomials in $n$.  The difference operator $\Knat(n,y\sS^{-2}_n)$,
  maps sequences with support in $\Nz$ to sequences with support in
  $\Ilam$.  Moreover,
  \begin{equation}
    \label{eq:HlamAnat}
    \kaplamn{N}(n)\Hlam_n(x,y) = \Knat(n,y\sS^{-2}_n) H_n(x,y),\quad n\in
    \Ilam,
  \end{equation}
\end{thm}
\noindent
It will be instructive to reformulate this result in a more explicit
manner.
The Theorem claims that
\begin{equation}
  \label{eq:upslamdef}
  \upslam_j(n) := \frac{\Knat_j(n)}{\kaplamn{N}(n)},\quad
  j=0,\ldots, N
\end{equation}
are polynomials with $\upslam_0(n)=1$.
It also claims that
\begin{equation}
  \label{eq:HlamHsum} \Hlam_n(x,y) = H_n(x,y) + \sum_{j=1}^N
  \upslam_{j}(n) y^j H_{n-2j}(x,y).
  n\in \Ilam.
\end{equation}
and that $\upslam_j(n)=0$ if $n\in \Ilam$ but $n-2j<0$.

\begin{lem}
  The operator
  $\Klam(z,yz^{2},\partial_z)\circ \partial_z^N$ maps $\C[z]$ into
  $z^N\Wlam\otimes \C[y]$.
\end{lem}
\begin{proof}
  Let
  \begin{equation}
    \label{eq:Kjdef}
    K_j(x,z):= \sum_{i=0}^{N-j} K_{ij} x^i z^{i+2j} 
  \end{equation}
  denote the coefficients of $z^N \Klam(x,y,z)$.  Observe that
  \[
    \Klam(z,yz^{2},\partial_z)\circ \partial_z^N
    = \sum_{j=0}^N
    \sum_{i=0}^{N-j}  K_{ij} y^j z^{i+2j} \partial_z^{i+2j}  
    = \sum_{j=0}^N y^j z^{2j} K_j(z,\partial_z)
  \]
  By \eqref{eq:Psilamgf}, \eqref{eq:KlamPsi}, $\Klam(\partial_z,y,z)$
  maps $\C[z]$ into $\Wlam\otimes \C[y]$.  It follows that each
  $K_j(\partial_z,z),\; j=0,\ldots, N$ maps $\C[z]$ into $z^N\Wlam$.
  Observe that $ K_j(z,x) z^{2j}= K_j(x,z) x^{2j}$.  Consequently,
  \begin{equation}
    \label{eq:AjzDz}
    z^{2j} K_j(z,\partial_z) =  K_j(\partial_z,z) \circ
    \partial_z^{2j},\quad j=0,\ldots, N  
  \end{equation}
  also maps $\C[z]$ into $z^N\Wlam$.  
\end{proof}

\begin{proof}[Proof of Theorem \ref{thm:Hlamlincomb}.]
  Let $K_j(x,z),\; j=0,\ldots, N$ be as in \eqref{eq:Kjdef}. By the
  preceding Lemma, for each $j=0,\ldots, N$, the operator
  $z^{2j}K_j(\partial_z,z)$ maps $\C[z]$ into
  $z^N\Wlam = \{ z^n \colon n\in \Ilam \}$.  Observe that
  \[ z^{2j}K_j(z,\partial_z) z^n = \Knat_j(n) z^{n},\quad j=0,\ldots,
    N.\] Since $\Nz = \cKlam_{N} \cup \Ilam$, it follows that
  $\Knat_j(n) =0$ if and only if $n\in \cKlam_{N}$.  This proves that
  each $\upslam_j(n),\; j=0,\ldots, N$ is a polynomial.  We already
  remarked that $\Klam(\partial_z,y,z)$ is a monic operator of order
  $N$. By \eqref{eq:AjzDz}, we have
  $\Knat_0(n) z^n = K_0(\partial_z,z) z^n$. Hence $\Knat_0(n)$ is a
  monic polynomial of degree $N$.  From this, it follows that
  $\upslam_0(n)=1$.  By \eqref{eq:AjzDz}, $\Knat_j(n)=0$ if $n<2j$.
  Therefore, $\upslam_j(n) = 0$ if $n\in \Ilam$ and $n-2j<0$.

  To prove \eqref{eq:HlamHsum}, observe that
  \begin{equation}
    \label{eq:gf0gflam}
    \begin{aligned} z^N \taulam(x,y)\Psilam(x,y,z)
      &= z^N     \Klam(\partial_z,y,z) \Psi_0(x,y,z)\\
      &= \sum_{n=0}^\infty      \sum_{j=0}^N y^jH_n(x,y)
      \Knat_j(n)\frac{z^{n}}{n!}
    \end{aligned}
  \end{equation}
Hence,
\[ \sum_{n\in \Ilam} \kaplamn{N}(n)\lp\sum_{j=0}^N y^jH_n(x,y)
  \upslam_j(n)\frac{z^{n}}{n!}\rp= \sum_{n\in \Ilam} \Hlam_n(x,y)
  \kaplamn{N}(n) \frac{z^n}{n!}. \]
\end{proof}

\subsection{Exceptional One-point functionals}
\label{sect:1ptcond}
Recall from \eqref{eq:Wlamdef} and \eqref{eq:Wslamdef} that $\Wlam$ is
the span of monomials corresponding to the Maya diagram $\Mlam$, and
that $\Wslam$ is the span of the exceptional Hermite rational
functions. In this section we will show that $\Wlam$ and $\Wslam$ have
a dual relation with respect to the bispectral involution on $\Grad$.
In effect, this serves as a characterization of the 1-point
functionals that annihilate the exceptional Hermite polynomials.  As a
byproduct, we will obtain a characterization of $\Flam$, the module of
$\lambda$-generators, in terms of 1-point functionals.

Although the semi-stationary $\Psilam(x,y,z)$ depends only on three
variables, it is also possible to write it in terms of the stationary
wave function depending only on two variables --- provided we
interpret the dependence on $y$ as a curve in $\Grad$.  For
$\cc\in \cC$, let $\hcc_y$ denote the 1-parameter family of
functionals defined by
\begin{equation}
    \label{eq:hcyz}
    \la \hcc_y(z),f(z)\ra  = \la \cof{z},  \e^{yz^2} f(z) \ra.
\end{equation}

In general, the coefficients of $\hcc_y$ involve exponential functions
of $y$.  However, for $c\in \cC_0$ the coefficients are polynomials;
that is, if $c=\bev{n}0$, then $\hc_y \in \C[y] \otimes \cC_0$. Explicitly,
by \eqref{eq:h2rr} and \eqref{eq:HBell},
\begin{equation}
  \label{eq:ehn0}
  \begin{aligned}
    \hbev{n}0_y
    &= \sum_{k=0}^n \binom{n}{k} H_k(0,y) \bev{n-k}0\\ 
    &= \sum_{j=0}^{\lfloor n/2 \rfloor} \frac{n!}{(n-2j)! j!} y^j
    \bev{n-2j}{0}
  \end{aligned}
\end{equation}
One can then extend \eqref{eq:ehn0} by linearity to all $\cC_0$.  For
$C\subset \cC_0$, let $\hC \subset \C[y] \otimes \cC_0$ be the
corresponding 1-parameter family of functionals. Let
\begin{equation}
  \label{eq:hWCdef}
    \hWlam_y := W_{\hClam_y}\quad \text{where }
  \hC_y = \{ \hcc_y \colon \cc \in C\}.
\end{equation}
be the corresponding curve in $\Grad$.

\begin{prop}
  With the above definitions, we have
  \begin{equation}\label{eq:addingy}
    \Psilam(x,y,z)=\psi_{\hWlam_y}(x,z) \e^{yz^2}.
  \end{equation}
\end{prop}
\begin{proof}
  By definition, $C_\lambda$ is spanned by
  $\bev{k_i}0,\; i=1,\ldots, \ell$.  By \eqref{eq:Psigf} and
  \eqref{eq:hcyz},
  \[ H_{n}(x,y) = \la c(z) , \exp(x z + y z^2) \ra = \la
    \hc_y(z) ,  \e^{x z} \ra,\quad c=\bev{n}0.\] Hence, by
  \eqref{eq:psixzdef},
  \[ \e^{yz^2} \psi_{\hWlam_y}(x,z) = \frac{ \Wr_x[H_{k_\ell}(x,y),
      \ldots, H_{k_1}(x,y),\e^{xz}] \e^{yz^2}}{\Wr_x[H_{k_\ell}(x,y),
      \ldots, H_{k_1}(x,y)] z^\ell}
    = \Psilam(x,y,z). \]
\end{proof}


Our next observation is the following characterization of $\Wlam$ as a
point in $\Grad$.
\begin{prop}
  \label{prop:Gradfixed}
  We have $\beta(\Wlam) = \Wlam$; i.e., $\Wlam\in\Grad$ is a fixed
  point of the bispectral involution.
\end{prop}
\begin{proof}
  By definition of the stationary wave function; c.f., \eqref{eq:PsiWdef}
  \eqref{eq:psixzdef},
  \begin{align*}
    \psi_{\Wlam}(x,z)
    & = z^{-\ell} \frac{\Wr_x[x^{k_1},\ldots, x^{k_\ell},
      \e^{xz}]}{\Wr_x[x^{k_1},\ldots, x^{k_\ell}]} \\
    &=\prod_{i<j}(k_i-k_j) x^{-|\lambda|} z^{-\ell} \e^{xz}
    \begin{vmatrix}
      x^{k_1} & k_1 x^{k_1-1} & \ldots &F_\ell(k_1) \,x^{k_1-\ell}\\
      \vdots & \vdots & \ddots & \vdots\\
      x^{k_{\ell}} & k_n x^{k_\ell-1} & \ldots
      &F_\ell(k_\ell)\,
      x^{k_\ell-\ell} \\
      1 & z & \ldots & z^\ell
      \end{vmatrix}
  \end{align*}
  where $F_j(a): = a(a-1) \cdots (a-j+1)$. By inspection, the
  coefficient of $z^j$ in the above determinant is a constant times
  $x^p$ where
  \[ p = \sum_{i=1}^\ell k_i- \frac12 \ell(\ell+1) + j  = |\lambda| - \ell+
    j.\]
  It follows that $\psi_{\Wlam}(x,z)$ is a linear combination of monomials
  of the form $(xz)^{j-\ell}\e^{xz},\; j=0,\ldots, \ell$.  Therefore,  
  $\psi_{\Wlam}(x,z) = \psi_{\Wlam}(z,x)$.
\end{proof}


Recall that
$\Wslam(x,y) = \lspan \{ \Rlam_m(x,y) \colon m\in \Jlam \} \otimes
\C[y]$.  For a fixed $y\in \C$., let $\Wslam_y$ denote the vector
space obtained by restricting $\Wslam$ to that particular value of
$y$.  Starting from \eqref{eq:WRHchar}, a straightforward calculation
shows that
\[ \Rlam_m(x,0) = x^m,\quad m\in \Jlam.\]
Thus, by proposition \ref{prop:Gradfixed},
\[\Wslam_0 = \hWlam_0 = \beta(\hWlam_0).\]
We now show that  $\Wslam_y\in \Grad$ for all $y$.  

\begin{prop}\label{prop:functionals}
  We have $\Wslam_y = \beta(\hWlam_y)$ for all $y\in \C$.
\end{prop}
\noindent
In other words, the exceptional Hermite polynomials specify, and are
determined by a homogeneous subspace of conditions. To be more
precise,
\begin{equation}
  \label{eq:betahWC*}
  \Uslam_y = \taulam_y      \beta(\hWlam_y),\quad y\in \C.
\end{equation}
where $\Uslam_y\subset \cP$ denotes the restriction to a
particular value of $y\in \C$.
\begin{proof}[Proof of Proposition \ref{prop:functionals}.]
  Fix a $y\in \C$.
  By
  \eqref{eq:psibetaW} and \eqref{eq:addingy},
  \[\psi_{\beta(\hWlam_y)}(z,x)= \Psilam(x,y,z) \e^{-yz^2} .\]
  By \eqref{eq:Philamdef} and \eqref{eq:Psimiwa}, for
  $y,z$ fixed
  $\taulam(x,y)\Psilam(x,y,z)$is either regular or has removable
  singularities for all $x\in \C$.  Hence, $\taulam_y
  \beta(\hWlam_y)$ is a polynomial subspace for every value of
  $y$.  Thus, it becomes possible to define
  $\Cslam$ as the space of continuous curves in
  $c_y\in\cC$ such that $c_y \in \Ann \taulam_y
  \beta(\hWlam_y)$ for all $y$. Hence, $\Psilam_y(z,x)\e^{-yz^2},\;
  y\in \C$ has the form shown in \eqref{eq:psiai} with
  $C=\Cslam_y$ and $q_C(z) =
  \taulam(z,y)$.  Moreover, by Proposition \ref{prop:Psiform}, $c\in
  \Cslam$ if and only if
  \[ \la c_y(x),\Philam(x,y,z)\e^{xz+yz^2}\ra\equiv 0,\quad y\in \C.\] By
  \eqref{eq:Psilamgf1}, for every $c\in \cC$, we have
  \[ z^N\la \cof{x}, \Philam(x,y,z)\e^{xz+yz^2}\ra = \sum_{n\in \Ilam}
    \kaplamn{N}(n)\;\la \cof{x}, \Hlam_n(x,y)\ra \frac{z^{n}}{n!}.\]
  By \eqref{eq:Philamdef} and \eqref{eq:Psimiwa}, the above relation
  is the Taylor series of an entire function.  Hence, $c\in\Cslam$
  if and only if
  \[ \la c_y(x), \Hlam_n(x,y)\ra = 0 ,\quad n\in \Jlam,\; y\in \C.\]
  It follows that $\Uslam_y = \Ker \Cslam_y,\; y\in \C$, as was to be
  shown.
\end{proof}
\noindent
As a particular case, by \eqref{eq:xopscaling}, the elements of
$\Cslam(-1/4)$ are the one-point differential functionals that
annihilate the univariate exceptional Hermite polynomials
$h^{(\lambda)}_n,\; n\in \Jlam$.  

We are now able to provide the following alternate characterization of
$\Flam$, the module of $\lambda$-generators.
\begin{prop}
  \label{prop:lamgen}
  A $\Phi(x,y,z) \in \C[x,y,z,z^{-1}]$ is a $\lambda$-generator if
  and only if $z^\ell\Phi(x,y,z)$ is a polynomial and if
  \begin{align}
    \label{eq:lamgenz}
    \la  c(z), z^\ell\Phi(x,y,z)\e^{xz+yz^2} \ra  &\equiv 0 \quad
    \text{for all } c\in  \Clam; \text{ and }\\
    \label{eq:lamgenx}
    \la  c_y(x), \Phi(x,y,z) \e^{xz+yz^2} \ra  &\equiv 0 \quad
    \text{for all } c\in  \Cslam,\; y\in \C.
  \end{align}
\end{prop}
\begin{proof}
  Suppose that
  \begin{equation}
    \label{eq:PhiJell} 
    \Phi(x,y,z) \e^{xz+yz^2} = \sum_{m\in \Jlam}  F_m(x,y)
    z^{m},\quad F_m \in \Uslam.
  \end{equation}
  Then, for $c\in \Clam$ we have
  \begin{align*}
    \la  c(z), z^\ell\Phi(x,y,z)\e^{xz+yz^2} \ra
    = \sum_{m\in \Jlam}  F_m(x,y)
    \la c(z),z^{m+\ell} \ra\equiv 0.
  \end{align*}
  Similarly, for $c\in \Cslam$, 
  \begin{align*}
    \la  c_y(x), \Phi(x,y,z) \e^{xz+yz^2} \ra
    = \sum_{m\in \Jlam}  \la c_y(x)F_m(x,y)\ra    z^{m+\ell}
    \equiv 0 
  \end{align*}
  Conversely, suppose that $z^\ell\Phi(x,y,z)$ is a polynomial and
  that \eqref{eq:lamgenz} and \eqref{eq:lamgenx} hold.  By
  \eqref{eq:lamgenz} it follows that
  \[ \Philam(x,y,z) \e^{xz+yz^2} = \sum_{m\in \Jlam}^\infty F_m(x,y)
    z^m\]
  where $F_n(x,y)$ are polynomials.  By Proposition
  \ref{prop:functionals} and by \eqref{eq:lamgenx}, each $F_m\in
  \Uslam$. 
\end{proof}

\section{Bispectrality}
\label{sect:bispect}
\subsection{The stabilizer algebras}
For a partition $\lambda$, let $\cAlam$ be the
algebra of differential operators that preserve $\Wlam$.  Let
$\cAslam$ be the dual algebra of differential operators that preserve
$\Wslam$.  To be more precise, $\pi\in \cAlam$ if and only if
\[ \pi(\partial_z,z) \Wlam(z) \subset \Wlam(z) \]
and $\sigma\in \cAslam$ if and only if
\[ \sigma(x,y,\partial_x) \Wslam(x,y)\subset \Wslam(x,y).\] In the
subsequent sections we will see that the operators and eigenvalues
associated with the exceptional Hermite bispectral triple belong to
certain commutative subalgebras of $\cAlam$ and $\cAslam$.

To gain a better understanding of $\cAlam$, we introduce basic
homogeneous operators whose action on a monomial either annihilates
that monomial or shifts its degree. Set
\begin{align}
  \label{eq:Glamqdef}
  \cGlam_k &:= (\Mlam+k)\cap \Jlam,\quad k\in \Z \\
  \label{eq:gamqdef}
  \gamlam_k(m) &:= \prod_{i\in \cGlam_k} (m-i),
\end{align}
Let $\Glam_k$ be the differential operator defined by
\begin{equation}
  \Glam_k(\partial_z,z) := z^{-k} \gamlam_{k}(z\partial_z),\quad k\in \Z,
\end{equation}
where \[ \E(\partial_z,z) = z\partial_z, \] is the Cauchy-Euler
operator in $z$.
\begin{prop}
  \label{prop:hRgen}
  The operator algebra $\cAlam$ is generated by $\E$ (the Cauchy-Euler
  operator) and by $\Glam_k,\; k \in \Z$.
\end{prop}
\begin{proof}
  By construction,
  \begin{align}
    \label{eq:Ezm}
    \E(\partial_z,z) z^m &= m z^m\\
    \label{eq:Gkm}
    \Glam_k(\partial_z,z) z^m &= \gamlam_{k}(m) z^{m-k}.
  \end{align}
  The Cauchy-Euler operator preserves $\Wlam$ because the latter is
  generated by monomials, and because \eqref{eq:Ezm} holds.  By
  \eqref{eq:Gkm}, we have $z^m\in \Wlam(z)$ and
  $z^{m-k} \notin \Wlam(z)$ if and only if $m\in \Jlam$ and
  $m-k\in\Mlam$.  The latter is true if and only if $m\in \cGlam_{k}$,
  if and only if $\gamlam_{k}(m)=0$. For this reason, $\Glam_k$
  restricts to a well-defined linear transformation of $\Wlam$.

  Conversely, since $\Wlam$ is generated by monomials, $\cAlam$ is
  generated by homogeneous operators $G(\partial_z,z)$ with the
  property that $G(\partial_z,z) z^m \in \Wlam(z),\; m\in \Jlam$.  If
  $G(x,z)$ has weighted-degree $k$, then it must be of the form
  \[ G(\partial_z,z) = z^{-k}G^{\natural}(z\partial_z), \]
  where $G^{\natural}(m)$ is a polynomial that vanishes on all
  $ m\in \cGlam_{k}$.  It follows that
  $G^{\natural}(m) = \gamlam_{k}(m) \alpha(m)$ for some polynomial
  $\alpha(m)$, and hence that
  \[  G(\partial_z,z) = \Glam_k(\partial_z,z) \circ
    \alpha(z\partial_z),\]
  as was to be shown.
\end{proof}

The sequence of integers
\begin{equation}
  \label{eq:glamqdef}
  \glam_q := \#\cGlam_q  = \deg \gamlam_q(m)
\end{equation}
is a key combinatorial signature of the partition $\lambda$ and is
intimately connected with the structure of the stabilizer rings
$\cRlam$ and $\cRslam$ that will be introduced in the following
section.  For now, we note the following symmetry property.
\begin{prop}
  \label{prop:glamq}
  For $q\in \Z$, we have
  \begin{equation}
    \label{eq:glamq-q}
    \glam_q = \glam_{-q}+q.
  \end{equation}
\end{prop}
\begin{proof}
  Without loss of generality, suppose that $q\geq 1$.
  Since $\lambda_i=0$ for $i>\ell$ we have that
  \[m_{{\ell}+q+j}(\lambda)+q = -{\ell}-j =
    m_{{\ell}+j}(\lambda),\quad j=1,2,\ldots.\] Hence, by
  \eqref{eq:kaplamdef} \eqref{eq:kaplamndef} and \eqref{eq:Glamqdef},
  we have
  \begin{align*}
    \cGlam_q &= \{ m_1+q,\ldots, m_{{\ell}+q}+q,\ldots,
    m_{\ell+q+j}+q,\ldots\} 
    \setminus
    \{m_1,\ldots, m_{\ell},\ldots, m_{\ell+j},\ldots\}\\
    &= \{ m_1+q,\ldots, m_{{\ell}+q}+q\} \setminus
      \{m_1,\ldots, m_{\ell}\}\\
    &= \lp \{ m_1+q,\ldots, m_{\ell}+q\} \setminus
      \{m_1,\ldots, m_{\ell}\}\rp \sqcup \{ m_{\ell+1}+q,\ldots, m_{\ell+q}+q \},
  \end{align*}
  where $\sqcup$ denotes a disjoint union and $m_i = m_i(\lambda)$.
  By similar reasoning,
  \begin{align*}
    \cGlam_{-q}
    &= \{ m_1-q,\ldots, m_{\ell}-q \} \setminus \{
      m_1,\ldots, m_{{\ell}+q}\}\\
    &= \{ m_1-q,\ldots, m_{\ell}-q \} \setminus \{
      m_1,\ldots, m_{\ell}\}\\
    \cGlam_{-q}+q
    &= \{ m_1,\ldots, m_{\ell} \} \setminus \{
      m_1+q,\ldots, m_{\ell}+q\}
  \end{align*}
  If $A,B$ are finite sets of equal cardinality, then $A\setminus B$
  and $B\setminus A$ also have equal cardinality. It follows that
  \[ \# \cGlam_{-q} = \# (\cGlam_{-q}+q) = \# \lp \{ m_1+q,\ldots,
    m_{\ell}+q\} \setminus \{m_1,\ldots, m_{\ell}\}\rp  = \#
    \cGlam_{-q}-q.\]
\end{proof}
\noindent
The following technical result is fundamental in the manipulation of
generating functions based on $\Flam$.
\begin{lem}[Reindexing Lemma]
  \label{lem:muqmq}
  We have
  \begin{equation}
    \label{eq:kapgamF}
    \kaplam(m)\gamlam_{k}(m) = \kaplam(m-k)\gamlam_{-k}(m-k)
    F_k(m+\ell),\quad k\in \Z.  
  \end{equation}
\end{lem}
\begin{proof}
  Fix $\lambda$ and let $m_i=m_i(\lambda),\; i\in \N$ be as per
  \eqref{eq:midef}.  By inspection, \eqref{eq:kapgamF} holds for
  $k=0$.  By the definition of the generalized falling factorial
  \eqref{eq:ffacdef},
  \begin{equation}
    \label{eq:FqF-q}
    F_{-k}(m) F_k(m+k) = 1.
  \end{equation}
  Thus \eqref{eq:kapgamF} is equivalent to
  \[   
    \kaplam(m-k)\gamlam_{-k}(m-k) = \kaplam(m)\gamlam_k(m)
    F_{-k}(m+\ell-k) ,\quad k\in \Z. \] Hence, no generality is lost
  if we assume that $k\geq 1$.
  Set\[ I_{k} = \{ m_1+k,\ldots, m_{{\ell}}+k\} \cap \{m_1,\ldots,
    m_{\ell}\} \] As was shown in the Proof to Proposition
  \ref{prop:glamq},
  \begin{gather*}
    \cGlam_k    = \lp\{ m_1+k,\ldots, m_{{\ell}}+k\} \setminus
    \{m_1,\ldots, m_{\ell}\}\rp \sqcup \{ m_{\ell+1}+k,\ldots, m_{\ell+k}+k\}\\
        \cGlam_{-k}
    = \{ m_1-k,\ldots, m_{\ell}-k \} \setminus \{
    m_1,\ldots, m_{{\ell}}\}
  \end{gather*}
  Hence,
  \begin{gather*}
    \gamlam_k(m) = \prod_{k\in \cGlam_k} (m-k) =
    \frac{\prod_{i=1}^{{\ell}+k} (m-m_i-k)}{ \prod_{k\in
        I_{q}} (m-k)}
    =
    \frac{\kaplamn{\ell+k}(m+\ell)}{ \prod_{k\in
        I_{q}} (m-k)}\\
    \gamlam_{-k}(m-k)  =
    \frac{\prod_{i=1}^{{\ell}} (m-m_i)}{ \prod_{k\in
        I_{q}} (m-k)} =
    \frac{\kaplam(m)}{ \prod_{k\in        I_{q}} (m-k)} \\
    \kaplam(m-k) F_k(m+\ell)  = \kaplamn{\ell+k}(m+\ell)
  \end{gather*}
  The desired relation \eqref{eq:kapgamF} follows immediately.
\end{proof}

The algebras $\cAlam$ and $\cAslam$ are closely related to
  $\Flam$.  For $\pi(x,y,z)$ that is polynomial in $x$ set
\begin{equation}
  \label{eq:Phipidef}
  \Phi_\pi(x,y,z):= \taulam(x,y) \e^{-xz-yz^2}
  \pi(\partial_z,y,z) \Psilam(x,y,z).
\end{equation}
Equivalently,
\[ \pi(\partial_z,y,z) \Psilam(x,y,z) = \frac{\Phi_\pi(x,y,z)}{
    \taulam(x,y) } \e^{xz+yz^2}.\]
Observe a $\lambda$-generator is a function of three variables:
$\Phi(x,y,z)$.  However a $\pi(x,z)\in\cAlam(x,z)$ is a function of
only two variables.  Thus, we need to consider linear combinations of
operators in $\cAlam$ with polynomial coefficients to establish the
isomorphism $\Flam(x,y,z) \simeq \cAlam(x,z)\otimes \C[y]$.  For the
sake of convenience, we will denote the latter simply as
$\Flam \simeq \cAlam \otimes \cP$.

\begin{prop}
  \label{prop:Phipi}
  If $\pi\in \cAlam$, then $\Phi_\pi\in \Flam$. Conversely, for each
  $\Phi\in \Flam$ there exists a $\pi\in \cAlam \otimes \cP$ such that
  $\Phi= \Phi_\pi$.
\end{prop}

\begin{proof}
  Suppose that $\pi = \rE$ is the Cauchy-Euler operator.  We have
  \[   
    \E(\partial_z,z) \Psilam(x,y,z)= \frac{
      \Phi_{\rE}(x,y,z)}{\taulam(x,y)} \e^{xz+yz^2} =
    \sum_{m=-\ell}^\infty 
    \frac{\kaplam(m)m}{(m+\ell)!}  \Rlam(x,y) z^m
  \]
  Hence, $\Phi_E \in \Flam$ by inspection.  Next, suppose that
  $\pi= \Glam_k$ for some $k\in \Z$.  By \eqref{eq:Gkm} and the
  reindexing Lemma,
  \begin{align*}
    \Glam_k(\partial_z,z) \Psilam(x,y,z)
    &= \sum_{m=-\ell}^\infty    \frac{\kaplam(m)\gamlam_k(m)}{(m+\ell)!}
      \Rlam_m(x,y) z^{m-k} \\
    &= \sum_{m=-\ell}^\infty    \frac{\kaplam(m-k)\gamlam_{-k}(m-k)
      F_k(m+\ell)}{(m+\ell)!} 
      \Rlam_m(x,y) z^{m-k} \\
    &= \sum_{m=-\ell}^\infty
      \frac{\kaplam(m)\gamlam_{-k}(m)}{(m+\ell)!} \Rlam_{m+k} z^m
  \end{align*}
  The last step is justified by the fact that if $m-k<-\ell$, then
  $F_{k}(m+\ell) = 0$.  Hence, $\Phi_\pi\in \Flam$ in this case also.
  The case of the general $\pi\in \cAlam$ now follows by Proposition
  \ref{prop:hRgen}. 

  Conversely, let $\Phi\in \Flam$ be given.  Set
  $\hPhi(x,y,z) =\bH(-y,z)\Phi(x,y,z)$. Hence, by Lemma \ref{lem:yz2},
  $\hPhi(\partial_z,y,z) \e^{xz+yz^2} = \Phi(x,y,z) \e^{xz+yz^2}$.  By
  Proposition \ref{prop:lamgen},
  \[ \la c_y(x), \Phi(x,y,z) \e^{xz+yz^2} \ra  = \la
    c_y(x),\hPhi(\partial_z,y,z) \e^{xz+yz^2} \ra \equiv 0 \]
  for all $c\in \Cslam$.  The kernel of $\Klam(\partial_z,y,z)$ 
  consists of
  \[\psi_c(y,z):= \la c_y(x), \e^{xz+yz^2} \ra,\quad c\in \Cslam.\]
  Observe that
  \[\hPhi(\partial_z,y,z) \phi_c(y,z) = \la
    c_y(x),\hPhi(\partial_z,y,z) \e^{xz+yz^2} \ra = 0,\quad c\in
    \Cslam.\] Hence, there exists a $\pi(\partial_z,y,z)$ such that
  $\hPhi(\partial_z,y,z) = \pi(\partial_z,y,z) \circ
  \Klam(\partial_z,y,z)$. Hence, by \eqref{eq:KlamPsi},
  \begin{align*}
    \taulam(x,y)\pi(\partial_z,y,z) \Psilam(x,y,z)
    &= \pi(\partial_z,y,z) \circ\Klam(\partial_z,y,z) \e^{xz+yz^2}\\
    &= \Phi(x,y,z) \e^{xz+yz^2}
  \end{align*}
\end{proof}

Similarly, for $\sigma(x,y,z)$ that is polynomial in $z$, set
\begin{equation}
  \label{eq:Phisigdef}
  \Phi_\sigma(x,y,z):= \taulam(x,y) \e^{-xz-yz^2}
  \sigma(x,y,\partial_x) \Psilam(x,y,z).
\end{equation}
\begin{prop}
  \label{prop:Phisig}
  If $\sigma\in \cAslam$, then $\Phi_\sigma \in \Flam$. Conversely for
  each $\Phi\in \Flam$ there exists a $\sigma\in \cAslam$ such that
  $\Phi= \Phi_\sigma$.
\end{prop}
\begin{proof}
  Let $\sigma\in \cAslam$ be given.
  Hence,
  \begin{align*}
    \sigma(x,y,\partial_x) \Psilam(x,y,z)
    &=\frac{\Phi_\sigma(x,y,z)}{z^\ell \taulam(x,y)} \e^{xz+yz^2}\\
    &=\sum_{m\in \Jlam} \sigma(x,y,\partial_x)\Rlam_{m}(x,y)
      \kaplam(m)\frac{z^{m}}{(m+\ell)!}. 
  \end{align*}
  By assumption,
  $\sigma(x,y,\partial_x)\Rlam_m(x,y)\in \Wslam(x,y)$.  Therefore,
  $\Phi_\sigma\in \Flam$.

  Conversely, let $\Phi\in \Flam$ be given.  By Proposition
  \ref{prop:lamgen},
  \[ \la c(z),z^\ell \Phi(x,y,z) \e^{xz+yz^2} \ra \equiv 0 \] for all
  $c\in \Clam$.  The kernel of $\Kslam(x,y,\partial_x)$ is spanned by
  $H_{k_1}(x,y),\ldots, H_{k_\ell}(x,y)$.  Let
  $c_i = \bev{k_i}0,\; i=1,\ldots, \ell$ and recall that
  \[ H_{k_i}(x,y) = \la c_i(z), \e^{xz+yz^2}\ra.\]
  Hence,
  \[ \Phi(x,y,\partial_x) \Hlam_{k_i}(x,y) = \la
    c_i(z),\Phi(x,y,\partial_x) \e^{xz+yz^2} \ra = 0,\quad
    i=1,\ldots, \ell.\] 
  Hence, there exists a $\sigma(x,y,\partial_x)$ such that
  \[ \Phi(x,y,\partial_x) = \taulam(x,y)\sigma(x,y,\partial_x)
    \circ \Kslam(x,y,\partial_x).\] Hence, by \eqref{eq:KlamPsi},
  \begin{align*}
    \taulam(x,y)z^\ell\sigma(x,y,\partial_x) \Psilam(x,y,z)
    &= \taulam(x,y)\sigma(x,y,\partial_x) \Kslam(x,y,\partial_x)
      \e^{xz+yz^2}\\ 
    &= \Phi(x,y,\partial_x) \e^{xz+yz^2}= \Phi(x,y,z) \e^{xz+yz^2}
  \end{align*}

\end{proof}
\begin{thm}
  \label{thm:flathom}
  For every $\pi \in \cAlam\otimes\cP$, there exists a
  $\pi^\flat\in \cAslam(x,y)$ such that
  \begin{equation}
    \label{eq:piflat} \pi(\partial_z,y,z) \Psilam(x,y,z) =
    \piflat(x,y,\partial_x) \Psilam(x,y,z).
  \end{equation}
  The corresponding mapping $ \cAlam\otimes \cP\to \cAslam$ is an
  algebra anti-isomorphism.
\end{thm}
\begin{proof}
  Let $\pi\in \cAlam\otimes \cP$ and let $\Phi_\pi$ be the polynomial
  given by \eqref{eq:Phipidef}. By Proposition \ref{prop:Phisig},
  there exists a $\piflat\in \cAslam$ such that
  \[ \piflat(x,y,\partial_x) \Psilam(x,y,z) = \frac{\Phi(x,y,z)}{z^\ell
      \taulam(x,y,z)} \Psi_0(x,y,z) = \pi(\partial_z,y,z) \Psilam(x,y,z).\]

  Given, $\pi_1,\pi_2\in \cAlam\otimes \cP$ observe that
  \begin{align*}
    (\pi_1(\partial_z,y,z) \circ \pi_2(\partial_z,z)) \Psilam(x,y,z)
    &= \pi_1(\partial_z,y,z) \lp \piflat_2(x,y,\partial_x)
      \Psilam(x,y,z)\rp\\
    &= \piflat_2(x,y,\partial_x) \lp \pi_1(\partial_z,y,z)
      \Psilam(x,y,z)\rp\\
    &= (\piflat_2(x,y,\partial_x) \circ \piflat_1(x,y,\partial_x))
      \Psilam(x,y,z)
  \end{align*}
  Therefore $\pi\mapsto \piflat$ is an anti-homomorphism.  By
  Proposition \ref{prop:Phipi} this anti-homomorphism is onto.
  Suppose that $\pi(\partial_z,y,z)$ is an operator that annihilates
  $\Psilam(x,y,z)$.  Hence, it must annihilate every
  $z^m,\; m\in \Jlam$.  This implies that $\pi=0$, and therefore
  $\pi\mapsto \piflat$ is also one-to-one.
\end{proof}
\noindent
See \cite{bakalov} for a similar use of anti-isomorphisms in the study
of bispectrality.

A fundamental instance of the anti-isomorphism \eqref{eq:piflat} is
the relation
\begin{equation}
  \label{eq:ET}
  \rE(\partial_z,z) \Psilam(x,y,z) = \tT(x,y) \Psilam(x,y,z)
\end{equation}
where $\rE$ is the Cauchy-Euler operator and where $\tT$ is the
exceptional operator given in \eqref{eq:tTdef}. An examination of the
generating function \eqref{eq:Psilamgf} shows that \eqref{eq:ET} is
equivalent to the eigenvalue equation \eqref{eq:TRmR}.  We will prove
the former and thereby establish the latter.

\begin{proof}[Proof of Proposition \ref{prop:TRmR}]
  Set $\displaystyle \ulam(x,y) = \lp\log \taulam(x,y)\rp_{xx}$. We
  wish to show that
  \[ \rEflat(x,z) = y z^2 + x z+ 4y\, \ulam(x,y).\]
  Since $\rE(x,z)=xz$, we have
  \[ \hrE(x,y,z) = \bH(y,z) (xz) = z H_1(x+2yz,y) = 2yz^2+xz.\]
  Hence,
  \[ \hrE(x,y,\partial_x) = 2y\partial_{x}^2 + x \partial_x =T(x,y) \]
  is the classical Hermite operator. 
  Hence, it suffices to show that
  \[ \tT(x,y)\circ \Kslam(x,y,\partial_x) =
    \Kslam(x,y,\partial_x) \circ T(x,y) ,\]
  which is equivalent to
  \[ [T,\Kslam] = -4 y \ulam \Kslam.\]
  It is well known that
  \[ \Wr[f_1,\ldots, f_l, f] = \tau f^{(l)} -\tau' f^{(l-1)} + \ldots
    ,\] where $\tau = \Wr[f_1,\ldots, f_l]$.  For the case at hand,
  $\Kslam = \partial_x^\ell - \frac{\taulam_x}{\taulam}
  \partial_x^{\ell-1} + \cdots$.  Hence, $[T,\Kslam]$ is an operator
  of order $\ell$ --- the same as the order of $\Kslam$.  By
  \eqref{eq:AlamWr}, the kernel of $\Kslam$ is generated by
  $H_{k_\ell},\ldots, H_{k_1}$.  These are all eigenfunctions of $T$,
  and hence annihilated by the commutator.  Hence,
  $ [T,\Kslam] =\mu \Kslam$. By inspection, $\mu$ is the leading
  coefficient of
  \[ -\left[2 y \partial_{x}^2, \frac{\taulam_x}{\taulam}
      \partial_x^{l-1}\right] = -4y   \ulam\partial_x^\ell.\] 
\end{proof}

\subsection{The bispectral triple.}
\label{sect:btriple}
\hlight{
For a partition $\lambda$ and $\Wlam\in \Grad$ as per
\eqref{eq:Wlamdef}, let $\cRlam\subset \cP$ denote the ring that
preserves $\Wlam$.
Let $\cRslam$ denote the ring\footnote{\hlight{Strictly speaking,
    $\cRlam$ has the structure of a $\C$-algebra and $\cRslam$ the
    structure of a $\cP$ algebra, but the accepted custom in Sato
    theory seems to be to refer to these objects as stabilizer
    rings.}} of bivariate polynomials that preserve $\Wslam$.}

Thus, $\pi\in \cRlam$ if and only if
\[ \pi(z)\Wlam(z)\subset \Wlam(z)\]
and  $\sigma \in \cRslam$ if and only if
\[\sigma(x,y) \Wslam(x,y) \subset \Wslam(x,y).\]
It is natural to regard
$\cRlam$ as a commutative subalgebra of
$\cAlam$ and to regard
$\cRslam$ as a commutative subalgebra of
$\cAslam$.  Indeed, as a direct consequence of Theorem
\ref{thm:flathom} we have
\begin{align}
  \label{eq:piflat1}
  \piflat(x,y,\partial_x) \Psilam(x,y,z)
  &= \pi(z)   \Psilam(x,y,z),\quad \pi \in \cRlam,\\
  \label{eq:sigsharp1}
  \sigsharp(\partial_z,y,z) \Psilam(x,y,z)
  &= \sigma(x,y)  \Psilam(x,y,z),\quad \sigma \in \cRslam.
\end{align}
Let
\begin{align*}
  \cSslam &= \{ \piflat \colon \pi\in \cRlam\},\quad\text{and}\quad
  \cSlam = \{ \sigsharp \colon \sigma \in \cRslam \}
\end{align*}
denote the corresponding commutative subalgebras of
$\cAslam$ and $\cAlam\otimes
\cP$, respectively.  Thus \eqref{eq:piflat1} and \eqref{eq:sigsharp1}
should be regarded as the eigenvalue equations of the bispectral
triple $(\cSlam, \cSslam,
\Psilam)$. This is, essentially, a parameterized version of Wilson's
construction in \cite{wilson}.

\begin{definition}
  We say that $q\in \Nz$ is a critical degree of $\cRlam$ if
  $z^q\in \cRlam(z)$.  Analogously, we say that $q\in
  \Nz$ is a critical degree of
  $\cRslam$ if there exists a $\sigma(x,y)\in
  \cRslam(x,y)$ such that $\deg_x\sigma(x,y) = q$.  Let
  $\cDlam$ denote the set of critical degrees of
  $\cRlam$ and
  $\cDslam$ denote the set of critical degrees of $\cRslam$.
\end{definition}
\noindent
Observe that since
$\cRlam,\cRslam$ are closed under composition, both
$\cDlam$ and $\cDslam$ are additive subsets of $\Nz$.

In Section \ref{sect:lowop}, below, we will show that the operators in
$\cSslam$ are lowering operators for exceptional Hermite polynomials
with \eqref{eq:piflat1} understood as a lowering relation.
In Section \ref{sect:rr} we will exhibit a homomorphism that
maps $\cSlam$ into a certain commutative algebra of difference
operators $\fSlam$. This homomorphism transforms 
the differential eigenvalue equation \eqref{eq:sigsharp1} into  the
discrete eigenvalue equation
\begin{align}
  \label{eq:Thetsigq}
  \Thetlam_q(m,y,\sS_m) \Rlam_m(x,y)
  &= \siglam_q(x,y) \Rlam_m(x,y),\quad q\in \cDslam.
\end{align}
where $\sigma_q\in \cRslam$ is a monic polynomial of degree
$q\in \cDslam$, and $\Theta_q \in \fSlam$ is the corresponding monic
difference operator of order $2q$ obtained by applying the above
homomorphism to $\sigsharp_q$. Up to an index shift,
\eqref{eq:Thetsigq} can be regarded as a recurrence relation for
exceptional Hermite polynomials:
\begin{equation}
  \label{eq:ThetlamHn}
  \Thetlam_q(n-N,y,\sS_n) \Hlam_n(x,y)
  = \siglam_q(x,y) \Hlam_n(x,y),\quad q\in \cDslam.
\end{equation}
\hlight{However it is more instructive to couple \eqref{eq:ThetlamHn} with the
eigenvalue equation \eqref{eq:THn} into a differential-difference
bispectral triple $(\tTlam,\fSlam,\Rlam)$ involving exceptional
Hermite rational functions.}

\begin{prop}
  \label{prop:psdegord}
  For a monic $\pi\in \cRlam$, the operator $\piflat(\partial_z,y,z)$
  is a monic differential operator whose order is equal to the degree
  of $\pi$. Dually, for a monic $\sigma\in \cRslam$, the expression
  $\sigsharp(\partial_z,y,z)$ is a monic differential operator whose
  order is equal to the $x$-degree of $\sigma(x,y)$.
\end{prop}
\begin{proof}
  These assertions follow by inspection of the construction of
  $\Phi_\pi$ and $\Phi_\sigma$ utilized in the proofs to Proposition
  \ref{prop:Phipi} and \ref{prop:Phisig}.
\end{proof}



Going forward, the following combinatorial description of $\cRlam$
will prove useful.
\begin{definition}
  For $q\in \N$, a Maya diagram $M\subset \Z$ is a $q$-core
  \cite[p. 12]{macdonald} \cite[p. 123]{noumi} if and only if if there
  exist $n_0,\ldots, n_{q-1} \in \Z$ such that
  \begin{equation}
    \label{eq:Mcore}
    M =
    \bigcup_{i=0}^{q-1} \{ mq +i \colon m \leq n_i \},
  \end{equation}
  We will also say that a partition $\lambda$ is a $q$-core if the
  corresponding $\Mlam$ is a $q$-core.
\end{definition}
\noindent
The following alternative characterization of a $q$-core is useful.
\begin{prop}
  \label{prop:MMq}
  A Maya diagram $M\subset \Z$ is a $q$-core if and only if
  $M\subset M+q$.
\end{prop}
\begin{proof}
  Let $M\subset \Z$ be a Maya diagram and let $q\in \N$. Define the Maya
  diagrams
  \[ M_i = \{ m\in \Z \colon q m + i \in M \},\quad i=0,\ldots, q
    -1.\]
  By definition, $M$ is a $q$-core if and only if each $M_i$ is a
  trivial Maya diagram; i.e., if $M_i\subset M_i+1$.
  Observe that
  \[ M = \bigcup_{i=0}^{q-1} (q M_i + i),\quad\text{and }\quad 
    M+q = \bigcup_{i=0}^{q-1} (q (M_i+1) + i).
  \]
  Hence, $M$ is a $q$-core if and only if each $M_i\subset M_i+1$ if
  and only if $M\subset M+q$.
\end{proof}

\begin{prop}
  \label{prop:gq-q}
  Let $\glam_q,\; q\in \Z$ be the integer sequence defined in
  \eqref{eq:glamqdef}.
  A partition $\lambda$ is a $q$-core if and only if $\glam_{-q}=0$,
  or equivalently, if and only if $\glam_q =q$.
\end{prop}
\begin{proof}
  The first criterion is  a direct consequence of the definition
  \eqref{eq:Glamqdef}. The second criterion follows by \eqref{eq:glamq-q}.
\end{proof}

\begin{prop}
  \label{prop:Rcore}
  We have $q\in\cDlam$  if and only if
  $\lambda$ is a $q$-core.
\end{prop}
\begin{proof}
  By definition \eqref{eq:Wlamdef}, $z^q \in \cRlam(z)$ if and only if
  $\Jlam+ q \subset \Jlam$.  The latter is true if and only if
  $\Mlam\subset \Mlam+q$.
\end{proof}



The result below is a characterization of the inclusion
$\cSlam\subset \cAlam\otimes \cP$.  This result will play a key role
in our study of exceptional recurrence relations in the sections that
follow.
\begin{prop}
  \label{prop:pi*crit}
  For $\pi\in \cAlam \otimes \cP$, define
  \begin{equation}
    \label{eq:hpidef}
    \hpi(x,y,z) = \bH(y,z)\pi(x,y,z).    
  \end{equation}   Then $\pi\in \cSlam$ if and only if
  \begin{equation}
    \label{eq:degzR*}
     \deg_z \hpi(x,y,z) \leq 0 .
  \end{equation}
  Moreover, if condition \eqref{eq:degzR*} holds, then
  $\pi=\sigsharp$, where $\sigma(x,y)$ is the leading coefficient of
  $\hpi(x,y,z)$; i.e.,
  \begin{equation}
    \label{eq:degzR*1}
    \sigma(x,y) = \lim_{z\to \infty} \hpi(x,y,z).
  \end{equation}
\end{prop}
\hlight{
\begin{proof}
  Let $\pi(x,y,z) \in \cAlam(x,z)\otimes \C[y]$ be given.  By Lemma
  \ref{lem:yz2},
  \[ \pi(\partial_z,y,z) \e^{xz+yz^2} = \hpi(x,y,z) \e^{xz+yz^2} =
    \hpi(x,y,\partial_x) \e^{xz+yz^2}.\]
  Consequently,
  \[ \Kslam(x,y,\partial_x) \hpi(x,y,\partial_x) =
    \sigma(x,y)\Kslam(x,y,\partial_x) .\]
  By definition, $\pi\in \cSlam$ if and only if
  $\sigma(x,y)=\piflat(x,y,z)$ is $z$-independent.  It follows that
  $\piflat(x,y,\partial_x)$ is a zeroth order differential operator if
  and only if
  \eqref{eq:degzR*} holds.
\end{proof}
}

In Proposition \ref{prop:Rcore}, above, we showed that $q$ is a
critical degree of $\cRlam$ if and only if the corresponding partition
is a $q$-core.  A priori, there does not seem to be a simple criterion
that describes the critical degrees of $\cRslam$.  However, it is
possible to give some necessary conditions.

\begin{prop}
  \label{prop:tausimp}
  If for  fixed $y\neq 0$, the polynomial $\taulam_y(x)$ has only simple
  zeroes, then $\sigma\in\cRslam$ if and only if $\sigma(x,y)$ is
  weighted-homogeneous and $\partial_x \sigma(x,y)$ is divisible by
  $\taulam(x,y)$.
\end{prop}
\noindent
The proof can be found in \cite{GKKM}.  Thus, if $\taulam$ has only
simple roots, then the critical degrees of $\cRslam$ consists of all
$q\geq N+1$ where $N=|\lambda| = \deg \taulam$.  In this case,
$\cRslam$ is the span of the following monic polynomials:
\begin{equation}
  \label{eq:sigidef}
  \sigma_q(x,y) :=\frac{1}{q+N} \int^x x^q \taulam(x,y),\quad
  q=0,1,2,\ldots 
\end{equation}

\begin{prop}
  \label{prop:sigcore}
  We have $\cDslam\subset \cDlam$.  In other words, if $q$ is a
  critical degree of $\cRslam$, then necessarily $\lambda$ is a
  $q$-core.
\end{prop}
\begin{proof}
  Let $q\in \cDslam$ and $\sigma_q(x,y)\in \cRslam(x,y)$ a
  corresponding eigenvalue such that $\deg_x \sigma_q(x,y)=q$.  Fix an
  $m\in \Jlam$.  By definition, $\sigma_q(x,y) \Rlam_m(x,y)$ is a
  $\C[y]$-linear combination of $\Rlam_{k}(x,y), k\in \Jlam$. By
  definition \eqref{eq:Rlamm}, $\Rlam_m(x,y)$ is monic in $x$.  By
  assumption, $\sigma_q(x,y)$ is also monic in $x$.  Hence
  $m+q \in \Jlam$ also.  Hence $\Jlam+q\subset \Jlam$.  Therefore
  $\Mlam$ is a $q$-core, by Proposition \ref{prop:Rcore}.
\end{proof}
The converse need not hold.  In Example \ref{ex:L22}, below, we will
demonstrate that the partition $\lambda=(2,2,0,\ldots)$ is a $4$-core,
but that $4$ is not a critical degree of $\cRslam$.

\section{Lowering and recurrence relations for the exceptional
  Hermites}
In Section \ref{sect:lowop}, we show that relation \eqref{eq:piflat1}
of the bispectral triple corresponds to a lowering relation for
exceptional Hermite polynomials.  In the following section, we exhibit
a homomorphism from differential operators to difference operators
that will allow us to transform \eqref{eq:sigsharp1} into recurrence
relations.  Thus, the existence of lowering operators in $\partial_x$
and difference operators in $n$ sharing the exceptional Hermites as
eigenfunctions follows naturally from the bispectrality of the wave
functions in the adelic Grassmannian.  Furthermore, it is notable that
the correspondence is \textit{constructive} in nature.  Consequently,
we will emphasize not only that the existence of the operators follows
from Wilson's construction but moreover that this provides a
convenient way to actually compute all of the corresponding operators.


\subsection{Lowering operators}
\label{sect:lowop}

 We are now ready to describe the
lowering operators for the exceptional Hermite rational functions
$\Rlam_n(x,y),\; n\in \Jlam$.  A conjugation of the lowering relation
\eqref{eq:LqPlamgam} below by $\taulam(x,y)$ then yields the
corresponding lowering relations for the exceptional polynomials
$\Hlam_n(x,y),\; n\in \Ilam$.

\begin{thm}
  \label{thm:Lqgamq}
  Let $q\in \N$ be a critical degree of $\cRlam$.  Let
  $\Llam_q (x,y,\partial_x)=\piflat(x,y,\partial_x)$ be the the
  corresponding monic operator of order $q$ corresponding to
  $\pi(z) = z^q$.  Then
  \begin{equation}
    \label{eq:LqPlamgam}
    \Llam_q(x,y,\partial_x) \Rlam_m(x,y) = \gamlam_q(m)
      \Rlam_{m-q}(x,y),\quad m\in\Jlam. 
    \end{equation}
    where $\Rlam_m(x,y),\; m\in \Jlam$ are the exceptional rational
    functions \eqref{eq:Rlamm}.  Moreover,
  \begin{equation}
    \label{eq:LqWronsk}
    \Llam_qf = \frac{\Wr[\Rlam_{k_1}, \ldots,
      \Rlam_{k_q}, f]}{\Wr[\Rlam_{k_1}, \ldots,
      \Rlam_{k_q} ]},
  \end{equation}
  where $k_1,\ldots, k_q $ is an enumeration of
  $\cGlam_q\subset\Jlam $.
\end{thm}
\begin{proof}
  Using \eqref{eq:kaplamndef}, \eqref{eq:kaplamdef},
  \eqref{eq:Psilamgf} we have:
  \begin{align*}
    \Llam_q(x,y,\partial_x) \Psilam(x,y,z)
    &=z^q \Psilam(x,y,z)\\
    &=\sum_{m=-{\ell}}^\infty \kaplam(m)\Rlam_m(x,y,z) 
      \frac{z^{m+q}}{(m+{\ell})!}\\
    &=\sum_{m=-{\ell}}^\infty \kaplamn{q}(m)\Rlam_{m-q}(x,y,z) 
      \frac{z^{m}}{(m+{\ell})!},
  \end{align*}
  where $\ell=\ell(\lambda)$.
  By \eqref{eq:Glamqdef}, we have
  \[ \{m_1+q,\ldots, m_q+q\} = \{ m_1,\ldots, m_{\ell}\} \cup
    \cGlam_q.\]
  Hence, $\kaplamn{q}(m) = \kaplam(m) \gamlam_q(m)$, which implies
  \eqref{eq:LqPlamgam}.  Furthermore, since $\gamlam_q(m)$ vanishes
  precisely at $k_1,\ldots, k_q$, it follows that
  $\Rlam_{k_1},\ldots, \Rlam_{k_q}$ are in the kernel of $\Llam_q$.  Since
  $\Llam_q$ is a monic differential operator, \eqref{eq:LqWronsk} follows.
\end{proof}

\begin{prop}
  The commutative algebra $\cSslam$ is generated by the lowering
  operators $\Llam_q,\; q\in \cDlam$.
\end{prop}
\begin{proof}
  Since $\Wlam$ is spanned by monomials, the same is true for
  $\cRlam$.
\end{proof}

It is also worthwhile to note the following, alternate,
characterization of the lowering operators.
\begin{prop}
  \label{prop:LqDxq}
  Let $\Llam_q(x,y,\partial_x),\; q\in \cDlam$ be a lowering operator
  \eqref{eq:LqWronsk} and $\Klam(x,y,\partial_x)$ the intertwiner as
  per \eqref{eq:Kslamdef}.  Then,
  \begin{equation}
    \label{eq:LqDxq}
    \Kslam(x,y,\partial_x) \circ \partial_x^q =
    \Llam_q(x,y,\partial_x) \circ \Kslam(x,y,\partial_x).
  \end{equation}
\end{prop}
\begin{proof}
  It suffices to observe that
  \begin{align*}
    z^{q+\ell} \Psilam(x,y,z)
    &=  \Kslam(x,y,\partial_x)(z^q \e^{xz+yz^2})\\
    &= \Kslam(x,y,\partial_x)\circ\partial_x^q \,\e^{xz+yz^2}\\
    &= L_q(x,y,\partial_x)     \Kslam(x,y,\partial_x) \e^{xz+yz^2}.
  \end{align*}
\end{proof}
\noindent
In this way, we recover the interpretation of $\cRlam$ and of critical
degrees presented in \eqref{eq:Lpdef}.

\subsection{Recurrence relations}
\label{sect:rr}
In this section we describe the
recurrence relations satisfied by exceptional Hermite polynomials.  As
a motivation, it is instructive to revisit the connection between the
classical 3-term recurrence relation \eqref{eq:class3term} and the first
order differential relation \eqref{eq:3termdiffrel}. Set
\[ \pi_1(x,y,z) := x - 2yz \] and express \eqref{eq:3termdiffrel} as
\[ \pi_1(\partial_z,y,z) \Psi_0(x,y,z) = x \Psi_0(x,y,z). \]
Set
\begin{equation}
  \label{eq:pinat1}
 \pinat_1(n,y,z) = -2yz+nz^{-1}    
\end{equation}
so that
\[ \pi_1(\partial_z,y,z)z^n = \pinat_1(n,y,z)z^n .\]
The classic recurrence relation \eqref{eq:class3term} can then be
derived as  follows:
\begin{align}
  \label{eq:xPsi0}
  x\Psi_0(x,y,z)
  &= \sum_{n=0}^\infty H_n(x,y)  \pinat(n,y,z) \frac{z^n}{n!}\\ \nonumber
  &= \sum_{n=0}^\infty H_n(x,y)  \lp -2y(n+1) \frac{z^{n+1}}{(n+1)!} +
    \frac{z^{n-1}}{(n-1)!}\rp\\
  &= \sum_{n=0}^\infty \Theta_1(n,y,\sS_n)H_n(x,y)  \frac{z^n}{n!}
\end{align}
where
\[ \Theta_1(n,y,z)= -2y\pinat_1(n,(4y)^{-1},z) = z-2ynz^{-1} .\]
As we will see below, the construction $\pi\to\pinat\to \Theta$
generalizes to the case of $\Psilam(x,y,z)$ and leads to an explicit
formula for exceptional recurrence relations.



Let $\lambda$ be a partition, $N=|\lambda|$ and
$\ell = \ell(\lambda)$.  As we will show below, the recurrence
relations corresponding to $\lambda$ take the form
\begin{equation}
  \label{eq:Thetsigq}
  \Thetlam_q(m,y,\sS_m)\Rlam_m(x,y)  =  \siglam_q(x,y) \Rlam_m(x,y) ,
\end{equation}
where $q\in \Nz$ is a critical degree of $\cRslam$, where
$\siglam_q\in \cRslam$ is a monic, homogeneous polynomial of degree
$q$, and where $\Thetlam_q(m,y,\sS_m)$ is a monic difference operator
of order $2q$ derived from the action of the corresponding $\pi_q =
\sigsharp_q$ on monomials $z^m$.
The transformation $\pi_q\mapsto
\Theta_q$ doubles the order
because the action of a $\pi_q(\partial_z,y,z),\; \pi\in
\cSlam,\; q\in \cDslam$ on $z^m$ involves degree shifts
$k\in \{ -q,-q+2,\ldots, q \}$.  In the classical case, this
phenomenon is illustrated by  relation \eqref{eq:pinat1}.

Thus, the operators $\Thetlam_q$ generate an algebra of difference
operators $\fSlam$ which is naturally isomorphic to the algebra of
differential operators $\cSlam$. This isomorphism is best understood
as the restriction of an algebra homomorphism $\cAlam \to \fAlam$,
where the latter is the algebra of difference operators that preserves
sequences with support in $\Jlam$.  This homomorphism from
differential to difference operators effectively transforms the
differential eigenvalue relation \eqref{eq:sigsharp1} into the
difference eigenvalue relation \eqref{eq:Thetsigq}.

Just like in the classical case, the exceptional Jacobi operator,
relative to a basis of normalized $\Rlam_m(x,y)$, is represented by a
symmetric matrix.  This is a consequence of the fact that
multiplication by the corresponding eigenvalue is a symmetric operator
relative to \eqref{eq:expx2ip}.  This symmetry imposes a certain
relation between the coefficients of the exceptional Jacobi operator
and the exceptional norming constants.  We will derive and make use of
this observation below.

We begin by describing the homomorphism $\cAlam \to \fAlam$.  For a
given partition $\lambda$, let $\cJslam$ be the vector space of
sequences supported on $\Jlam\subset \Z$.  Let $\varepsilon$ denote
the multiplication operator
\begin{equation}
  \label{eq:Emdef}
  \varepsilon(m) f_m = m f_m,
\end{equation}
where $f_m,\; m\in \Jlam$ is a sequence.
Evidently, $\varepsilon\in \End \cJslam$.  For $q\in \Z$, define the
weighted shift operator
\begin{equation}
  \label{eq:Gamqdef}
  \Gamlam_q(m,\sS_m) :=  \gamlam_{-q}(m) \sS_m^q.
\end{equation}
Observe that $\Gamlam_q\in \End \cJslam$ because
\[ \Gamlam_q(m,\sS_m) f_m = \gamlam_{-q}(m) f_{m+q},\] and because
$\gamlam_{-q}(m)=0$ precisely when $m\in \Jlam$ but $m+q\notin \Jlam$.
Let $\fAlam\subset \End \cJslam$ be the algebra of difference
operators generated by $\varepsilon$ and $\Gamlam_q,\; q\in \Z$.

Recall that $\Flam$ is isomorphic to the module of Laurent series
\[ \frac{\Phi(x,y,z)}{\taulam(x,y)} \e^{xz+yz^2} = \sum_{m\in \Jlam}
  F_m(x,y) z^m,\quad F_m\in \Wslam,\quad \Phi\in \Flam.\] Thus,
$\fAlam$ also acts on $\Flam$. This gives the isomorphism
$\cAlam\simeq \fAlam$ with $\rE(\partial_z,z)\mapsto \varepsilon(m)$
and $\Glam_q(\partial_z,z)\mapsto \Gamlam_q(m,\sS_m)$.

The homomorphism $\cAlam\to \fAlam$ can also be described as a mapping
$\pi(\partial_z,y,z) \mapsto \Thetlam_\pi(m,y,\sS_m)$ where the latter
will be defined in \eqref{eq:piflatpinat}, below.
\begin{prop}
  \label{prop:pinatzm}
  For every $\pi(x,y,z)\in \C[x,y,z,z^{-1}]$ there exists a
  $\pinat(m,y,z)\in \C[m,y,z,z^{-1}]$ such that
  \begin{equation}
    \label{eq:pinatzm}
    \pi(\partial_z,y,z) z^m = \pinat(m,y,z) z^m
  \end{equation}
\end{prop}
\begin{proof}
  It suffices to observe that $\partial_z^i z^m = F_i(m
  )z^{m-i}$. Thus, the mapping $\pi\mapsto \pinat$ is described by
  \[ x^i\mapsto F_i(m)z^{-i},\quad y^j\mapsto y^j,\quad z^k\mapsto
    z^k,\] where $F_i$ is the falling factorial \eqref{eq:ffacdef}.
\end{proof}
\noindent
Thus, if $\pi(x,y,z) \in \C[x,y,z,z^{-1}]$ is weighted-homogeneous of
degree $q\in \Z$ with $d=\deg_y \pi(x,y,z)$, then it can be given as
\begin{equation}
  \label{eq:pijcoeff}
  \begin{aligned}
    \pi(x,y,z) &= y^{q/2}\sum_{k} \pinat_{k}(xz) y^{-k/2}z^{-k} =
    \sum_{j=0}^d \pinat_{q-2j}(xz) y^j z^{2j-q}
  \end{aligned}
\end{equation}
where $\pinat_{q-2j}(x),\; j=0,\ldots, d$ are polynomials.  In this
way,
\begin{equation}
  \label{eq:pinatyj}
  \pinat(m,y,z) = y^{q/2} \sum_{k} \pinat_{k}(m) y^{-k/2} z^{-k}
  = \sum_{j=0}^d \pinat_{q-2j}(m) y^j z^{2j-q},\quad k =q-2j.
\end{equation}

\begin{prop}
  \label{prop:pinatjgamlam}
  For a weighted-homogeneous $\pi(x,y,z)$ we have
  $\pi \in \cAlam \otimes \cP$ if and only if every $\pinat_{k}(m)$ is
  divisible by $\gamlam_{k}(m)$.
\end{prop}
\begin{proof}
  By Proposition \ref{prop:hRgen} there exist polynomials
  $\al_k(m),\; k=q-2d,q-2d+2,\ldots, q$ such that
  \[ \pi(\partial_z,y,z) = y^{q/2}\sum_{k} \al_{k}(z\partial_z)
    \Glam_{k}(\partial_z,z) y^{-k/2} . \] Hence,
  \[ \pinat(m,y,z) = y^{q/2}\sum_k \al_{k}(m) \gamlam_{k}(m) y^{-k/2}
    z^{-k},\] and $\pinat_k(m) = \gamlam_k(m) \al_k(m)$.
\end{proof}

\begin{prop}
  \label{prop:Thetpi}
  For every $\pi \in \cAlam\otimes \cP$ we have
  \begin{equation}
    \label{eq:piTheta}
    \piflat(x,y,\partial_x) \Rlam_m(x,y) = \Thetlam_\pi(m,y,\sS_m)\Rlam_m(x,y).
  \end{equation}
  where  the difference operator
  $\Thetlam_\pi \in \fAlam\otimes \cP$ is given by
  \begin{equation}
    \label{eq:piflatpinat}
    \Thetlam_\pi(m,y,\sS_m) = y^{q/2}\sum_{k}
    \gamlam_{-k}(m) \al_{k}(m+k)   y^{-k/2}    \sS_m^{k} ,
  \end{equation}
  and where
  \begin{equation}
    \label{eq:alphakdef}
     \al_{k}(m) =\frac{\pinat_k(m)}{\gamlam_k(m)},\quad
    k=q-2d+q,q-2d+2,\ldots, q.
  \end{equation}
\end{prop}

\begin{proof}[Proof of Proposition \ref{prop:Thetpi}]
  Using \eqref{eq:Psilamgf}, \eqref{eq:piflat} \eqref{eq:kapgamF},
  we have
  \begin{align*}
    \sum_{m=-\ell}^\infty
    &   \piflat(x,y,\partial_x)\Rlam_m(x,y)
      \frac{\kaplam(m)}{(m+\ell)!} z^m\\
    &= y^{q/2}\sum_{m=-\ell}^\infty  \sum_{k}\frac{\kaplam(m)}{(m+\ell)!}
      \gamlam_{k}(m)\al_{k}(m)    \Rlam_m(x,y) y^{-k/2}     z^{m-k},\\
    &= y^{q/2}\sum_{m=-\ell}^\infty      \sum_{k}
      \gamlam_{-k}(m)\al_{k}(m+k)
      \Rlam_{m+k}(x,y)
      \frac{\kaplam(m)}{(m+\ell)!} y^{-k/2}     z^{m},
  \end{align*}
  where the last step is justified by the fact that if $m-k<-\ell$,
  then $F_{k}(m+\ell) = 0$.
  Matching coefficients yields
  \eqref{eq:piflatpinat}.
\end{proof}
\noindent
Note that the coefficients of $\Thetlam_\pi$ are non-singular because,
by Proposition \ref{prop:pinatjgamlam}, $\al_{k}(m)$ are polynomial.
Also note that $\Thetlam_\pi$ is an endomorphism of $\cJslam$, because
$m\in \Jlam$ and $m+k\notin \Jlam$ is precisely the condition
$\gamlam_{-k}(m)=0$.

Let $\fSlam \subset \fAlam\otimes \cP$ be the commutative subalgebra
corresponding to the image of $\cSlam\subset \cAlam\otimes \cP$ under
the above isomorphism.  The elements of $\fSlam$ are precisely the
exceptional Jacobi operators.  To be more precise, let $q$ be a
critical degree of $\cRslam$ and $\siglam_q(x,y)\in \cRslam(x,y)$ a
corresponding weighted homogeneous, $x$-monic polynomial of degree
$q$.  Let $\pi_q=\sigsharp_q$, so that \eqref{eq:sigsharp1} holds.  We
will refer to the corresponding difference operator
$\Thetlam_q := \Thetlam_\pi$ as a $q\supth$ order exceptional Jacobi
operator\footnote{There isn't a unique $q$th order Jacobi operator,
  because one can modify $\sigma_q(x,y)$ by adding eigenvalues of
  lower degree.}.  We are now able to assert the following.
\begin{thm}
  \label{thm:sigqRlam}
  Let $q$ be a critical degree of $\cRslam$, and
  $\siglam_q, \Thetlam_q,\pi_q$  as above so that, by definition,
  \begin{equation}
    \label{eq:sigqRlam}
    \siglam_q(x,y) \Rlam_m(x,y) = \Thetlam_q(m,y,\sS_m) \Rlam_m(x,y).
  \end{equation}
  Then, $\deg_y \pi_q(x,y,z) \le q$ and
  \begin{equation}
    \label{eq:Thetlamqpinat}
    \begin{aligned}
      \Thetlam_q(m,y,z)
      &=(-2y)^q\pinat_q\lp m,(4y)^{-1}, z\rp 
    \end{aligned}
  \end{equation}
  where $\pinat_q$ is related to $\pi_q$ by \eqref{eq:pinatzm}.
  Explicitly,
  \begin{align}
    \label{eq:sigRlam}
    \sigma_q(x,y) \Rlam_m(x,y)
    &=y^{q/2} \sum_k (-2)^{-k}  \pinat_{q,-k}(m) y^{-k/2}
    \Rlam_{m+k}(x,y) 
  \end{align}
  where $\pinat_{q,k}(m),\; k=-q,-q+2,\ldots, q$ are the coefficients
  of $\pinat_q$ as per \eqref{eq:pinatyj}.
\end{thm}

\begin{lem}
  \label{lem:degpigam}
  Let $d=\deg_y\pi_q(x,y,z)$ and let $\pinat_{q,k}(m)$ be the
  coefficients of $\pinat_q(m,y,z)$ as per \eqref{eq:pinatyj}. Then,
  necessarily $d\le q$ and $\deg\pinat_{q,k}(m) \leq (q+k)/2$ for all
  $k$.
\end{lem}
\begin{proof}
  By assumption, $\pi_q \in \cAslam\otimes \cP$.  Let
  $\hpi_q(x,y,z)= \bH(y,z) \pi_q(x,y,z)$.
  Let $\hpi_{ij}$ and $\pi_{ij}$ denote the corresponding
  coefficients so that
  \[
    \pi_q(x,y,z) = \sum_{i,j\ge 0} \pi_{ij} x^iy^j z^{i+2j-q}=
    \sum_{i,j\ge 0} \hpi_{ij} y^j z^{i+2j-q} H_i(x-2yz,-y).
  \]
  By Proposition \ref{prop:pi*crit}, $\deg_z \hpi_q(x,y,z)\le
  0$. Hence, there are no terms of positive $z$-degree in the last sum.
  Hence, $i+2j-q\geq j$ for all non-zero terms in the first sum.
  Hence, by \eqref{eq:pijcoeff},
  \[ \deg \pinat_{q,k}-k \leq (q-k)/2 \]
  for all $k\in \{ q-2d,q-2d+2,\ldots, q\}$.   In particular $2d-q
  \leq d$, which means that $d\leq q$.
\end{proof}

As was mentioned earlier, the symmetry of the exceptional Jacobi
operator imposes a certain relation between the coefficients
$\pinat_{k}(m),\; k=-q,-q+2,\ldots,q$ and the exceptional norming
constants $\nulam_m(y)$, as defined in \eqref{eq:nulamdef}.

\begin{lem}
  Let $q,\pi_q$ be as above, and let $\al_{q,k}(m)$ be as
  per \eqref{eq:alphakdef}.  Then,
  \begin{equation}
    \label{eq:alphasym}
    \al_{q,-k}(m) = (-2)^k\al_{q,k}(m+k),\quad k=-q,-q+2,\ldots, q.
  \end{equation}
\end{lem}
\begin{proof}
  By \eqref{eq:piflatpinat} and \eqref{eq:kapgamF}, for $m\in \Jlam$, we have
  \[ \sigma_q(x,y) \Rlam_m(x,y) = y^{q/2}\sum_{k}
    \gamlam_{-k}(m)\al_{k}(m+k) y^{-k/2} \Rlam_{m+k}(x,y).\] Hence,
  setting $n=m+k$ and using \eqref{eq:Rlamorth}, we have
  \begin{align*}
    \la \sigma_q \Rlam_m , \Rlam_{n}\ra_H
    &=
      \gamlam_{-k}(m) \al_{q,k}(m+k)  y^{(q-k)/2}
      \nulam_{n}(y)\\ 
    &=
      \gamlam_{m-n}(m) \al_{q,n-m}(n)  y^{(q+m-n)/2}
      \frac{(n+\ell)!}{\kaplam(n)} (-2y)^{n}
      \nulam_{0}(y)
  \end{align*}
  Since multiplication by $\sigma_q$ is a symmetric operator,
  the above expression is symmetric in  $m,n$. Hence,
  \begin{gather*}
    \gamlam_{m-n}(m) \al_{q,n-m}(m) \frac{(n+\ell)!}{\kaplam(n)}
    (-2)^{n} = \gamlam_{n-m}(n) \al_{q,m-n}(m)
    \frac{(m+\ell)!}{\kaplam(m)} (-2)^{m}\\
    \al_{q,-k}(m) = (-2)^k\al_{q,k}(m+k)
    \frac{\kaplam(m)\gamlam_{-k}(m)}{\kaplam(n)\gamlam_k(m+k)}F_k(m+k+\ell)
    \end{gather*}
    The desired relation now follows by \eqref{eq:kapgamF}.
\end{proof}

\begin{proof}[Proof of Theorem \ref{thm:sigqRlam}]
  By \eqref{eq:piflatpinat} and \eqref{eq:alphasym}, for $m\in \Jlam$
  we have
  \begin{equation}
    \label{eq:sigRgambet}
    \begin{aligned}
      \sigma_q(x,y) \Rlam_m(x,y)
      &= y^{q/2} \sum_{k}
      \gamlam_{-k}(m)\al_{q,k}(m+k) y^{-k/2} R_{m+k}(x,y) \\
      &= y^{q/2} \sum_k (-2)^{-k}\gamlam_{-k}(m) \al_{q,-k}(m) y^{-k/2}
      R_{m+k}(x,y). 
    \end{aligned}
  \end{equation}
  where the sum is over $k= -q,-q+2,\ldots, q$.
  On the other hand,
  \begin{align*}
    (-2y)^q\pinat_q(m,(4y)^{-1}, z)
    &= (-2 y)^q \sum_k    \pinat_{q,k}(m) (4y)^{(k-q)/2} z^{-k} \\
    &=  y^{q/2}\sum_k  (-2)^{-k}  \gamlam_{-k}(m) \al_{q,-k}(m)y^{-k/2} z^{k} 
  \end{align*}
  A direct comparison of the last line and of \eqref{eq:sigRgambet}
  establishes \eqref{eq:sigqRlam}.
\end{proof}

\section{Algorithms and examples}
\label{sect:algex}


\subsection{Intertwiners}
Recall that, by \eqref{eq:Rlamm} and \eqref{eq:AlamWr},
\[ \Rlam_m(x,y) = \frac{\Hlam_{m+N}(x,y)}{\taulam(x,y)} ,\quad m\in
  \Jlam \] may be given in terms of a Wronskian as
\[ \kaplam(m)\Rlam_m(x,y) = \Kslam(x,y,\partial_x)
  H_{m+\ell}(x,y),\quad m\in \Ilam.\] Theorem \ref{thm:Hlamlincomb}
exhibits a constructive procedure for giving exceptional Hermite
polynomials as linear combinations of classical polynomials by using
the dual intertwiner $\Klam(\partial_z,y,z)$.  Let us illustrate the
calculations with an example.

\begin{example}
  \label{ex:L22}
Consider the partition $\lambda = (2,2,0,\ldots)$.  Correspondingly,
$\ell=2,\;N=4$, and 
\[ \Mlam = \{1,0,-3,-4,\ldots \},\quad \Ilam = \{ 2,3,6,7,8,\ldots \},
  \quad \cKlam(\lambda) = \{2,3\}.\] Using \eqref{eq:Slamwronsk},
\eqref{eq:HBell}, and \eqref{eq:Philamdef}, we have
\begin{align*}
  \Slam(t_1,t_2,\ldots)
  &= \frac{t_1^4}{12} + t_2^2 - t_1 t_3\\
  \Philam(x,y,z)
  &=\lp x-z^{-1}\rp^4 + 12 \lp y-\frac12z^{-2}\rp   - \lp
    x-z^{-1}\rp \lp-\frac13 z^{-3}\rp\\
  &=x^4+12y^2 - 4x^3z^{-1} + (6x^2-12y)z^{-2},\\
  \taulam(x,y)
  &=  x^4 + 12 y^2 .
\end{align*}
Applying \eqref{eq:Klamdef} gives
\begin{align*}
  \Klam(\partial_z,y,z) &= \partial_z^4 - \lp 8 y z + 4z^{-1}\rp
                          \partial_z^3 +\lp 24 y^2 z^2 +12 y +
                          6z^{-2}\rp \partial_z^2 \\
                        &\qquad - 
                          32 y^3 z^3 \partial_z + \lp 16 y^4 z^4- 16 y^3 z^2 -
                          24yz^{-2}\rp.
\end{align*}
Applying \eqref{eq:Anatdef} gives
\begin{align*}
  \kaplamn{N}(n) &= n(n-1)(n-4)(n-5)\\
  \upslam_1(n)&=- 4 (2n-3)\\
  \upslam_2(n)&=24(n-2)(n-3)\\
  \upslam_3(n)&=-16 (n-2)(n-3)(2n-11)\\
  \upslam_4(n)&=16(n-2)(n-3)(n-6)(n-7)
\end{align*}
The corresponding exceptional polynomials
\[ \Hlam_n =\frac{\Wr[H_2,H_3,H_{n-2}]}{(n-4)(n-5)} ,\quad n\in
  \Ilam,\]
may therefore be given as
\begin{align*}
 \Hlam_n(x,y)
  &= H_n(x,y) - 4 (2n-3) y H_{n-2}(x,y)
  +24(n-2)(n-3)y^2H_{n-4}(x,y)\\
  &\qquad -16 (n-2)(n-3)(2n-11)y^3 H_{n-6}(x,y)\\
  &\qquad +16(n-2)(n-3)(n-6)(n-7)y^4 H_{n-8}(x,y),\quad n\in \Ilam.
\end{align*}

\end{example}

\begin{example}
  \label{ex:L21}
Next, consider the partition
$\lambda=(2,1,0,\ldots)$. Correspondingly,
\begin{equation}
  \label{eq:Mlam21}
  \Mlam = \{ \ldots, -4,-3,-1,1 \},\quad \Ilam = \{ 1,3,5,6,7,\ldots
  \},\quad \Klam = \{ 1,3 \}.
\end{equation}
By \eqref{eq:Hlamdef}, the corresponding exceptional polynomials are
\[ \Hlam_n = \frac{\Wr[H_1,H_3,H_{n-1}]}{2
    (n-2)(n-4)},\quad n\in \Ilam.\]
Using the same formulas as above, we have
\begin{equation}
  \label{eq:L21SPhitau}
  \begin{aligned}
    \Slam(t_1,t_2,\ldots)
    &= \frac{t_1^3}{3} - t_3\\
    \Philam(x,y,z)
    &= x^3 - 3 x^2z^{-1} + 3xz^{-2}\\
    \taulam(x,y) &= x^3.
  \end{aligned}
\end{equation}
The $\tau$-function of this example is degenerate because it
corresponds to a solution of KdV; the corresponding $\Wlam$ is a
stationary point of the second KP flow.  Applying \eqref{eq:Klamdef}
gives
\[ \Klam(x,y,z) = x^3 - \lp 6yz + 3z^{-1}\rp x^2 + \lp 12y^2 z^2 +
  6y+3z^{-2}\rp x - 8y^3 z^3\]
Note that since $\Wslam$ is stationary, we have
\[ \Kslam(x,y,z) = z^{-1}\Klam(z,0,x) = z^2 - 3x^{-1} z +  3 x^{-2} \]
Applying \eqref{eq:Anatdef} and \eqref{eq:HlamAnat} gives 
\[\Hlam_n = H_n+ 6yH_{n-2}-12(n-1)(n-3)y^2 H_{n-4} + 8 (n-1)(n-3)(n-5)
  y^3 H_{n-6}.\]
\end{example}


\subsection{Lowering operators}
In this section, we collect some calculations related to Theorem
\ref{thm:Lqgamq}.
\begin{continueexample}{ex:L21}
  Recall that $\lambda = (2,1,0,\ldots)$ with the corresponding Maya
  diagram given in \eqref{eq:Mlam21}.  We will use Proposition
  \ref{prop:Rcore} to determine the critical degrees of $\cRlam$.  The
  index set for $\Rlam_m(x,y),\; m\in \Jlam$ is
\[ \Jlam = \{ -2,0,2,3,4,5,6,\ldots \}.\]
\[
  \Rlam_{-2} = x^{-2},\quad
  \Rlam_0 = 1-6x^{-2}y,\quad
  \Rlam_2 = x^2-4y+12x^{-2}y^2,\quad
  \Rlam_3 = x^3, \;\ldots
\]
  \begin{figure}[ht]
  \centering
  \begin{tikzpicture}[scale=0.5]

    \fill  (-2.5,7.5) 
    \ncirc0;
    \fill[red]  (-2.5,7.5)   \ncirc2\ncirc2;

    \fill  (-2.5,6.5)
    \ncirc0\ncirc1 \ncirc2\ncirc2;

    \fill  (-2.5,5.5)  \ncirc0\ncirc1 ;
    \fill[red]  (-2.5,5.5)   \ncirc2\ncirc2\ncirc2;

    \fill  (-2.5,4.5)    \ncirc0\ncirc1
    \ncirc2\ncirc2\ncirc2 ;
    \fill[red]  (-2.5,4.5) \ncirc2\ncirc5;

    \fill  (-2.5,3.5)  \ncirc0 \ncirc1\ncirc2 ;
    \fill[red]  (-2.5,3.5)   \ncirc2\ncirc2\ncirc2\ncirc2;

    \fill  (-2.5,2.5)    \ncirc0 \ncirc1\ncirc2 \ncirc2 ;
    \fill[red]  (-2.5,2.5) \ncirc2\ncirc2\ncirc3\ncirc2;

    \fill  (-2.5,1.5) \ncirc0 \ncirc1\ncirc2 \ncirc2 ;
    \fill[red]  (-2.5,1.5) \ncirc2\ncirc2\ncirc2\ncirc2\ncirc2;

    \fill  (-2.5,0.5)  \ncirc0 \ncirc1\ncirc2 \ncirc2 ;
    \fill[red]  (-2.5,0.5) \ncirc2\ncirc2\ncirc2\ncirc1\ncirc2\ncirc2;

    \path (10.5,7.5) node[anchor=west] {$\Mlam-1$};
    \path (10.5,6.5) node[anchor=west] {$\Mlam$};
    \path (10.5,5.5) node[anchor=west] {$\Mlam+1$};
    \path (10.5,4.5) node[anchor=west] {$\Mlam+2$};
    \path (10.5,3.5) node[anchor=west] {$\Mlam+3$};
    \path (10.5,2.5) node[anchor=west] {$\Mlam+4$};
    \path (10.5,1.5) node[anchor=west] {$\Mlam+5$};
    \path (10.5,0.5) node[anchor=west] {$\Mlam+6$};

    \foreach \y in {0.5,...,7.5} \fill (-4.5,\y) \ncirc0\ncirc1;

    \draw  (-5,0) grid +(14 ,8);
    \draw[line width=2pt] (-3,0) -- ++ (0,8);

  \foreach \x in {-6,...,7} \draw (\x+1.5,-0.5)  node {$\x$};

\end{tikzpicture}
\label{tab:L21}
\caption{Translates of $\Mlam$ where $\lambda = (2,1,0,\ldots)$.}
\end{figure}

\noindent
By inspection of Table \ref{tab:L21}, $\cDlam=\{2,4,5,6,\ldots\}$,
which means that the ring of lowering operators $\cSslam$ is generated
by $L_2$ and $L_5$.

Applying \eqref{eq:LqWronsk} gives
\begin{align*}
  L_2 &=   \partial_{x}^2-6x^{-2}\\
  L_5 &= \partial_{x}^5 - 15 x^{-2} \partial_{x}^3 + 45 x^{-3}
        \partial_{x}^2 - 45 x^{-4} \partial_x
\end{align*}
Because $\Wlam$ is stationary under the 2nd KP flow, the lowering
operators are independent of $y$.  The corresponding lowering
relations are:
\begin{align*}
  L_2 \Rlam_m &= (m+2)(m-3) \Rlam_{m-2},& m\in \Jlam\\
  L_5 \Rlam_m &= (m+2)m(m-2)(m-4)(m-6) \Rlam_{m-5} & m\in \Jlam
\end{align*}
Note that the above relations are sensible, because the polynomial
$\gamlam_q(m)$ on the RHS annihilates precisely those indices
$m\in \Jlam$ for which $m-q\notin \Jlam$.
\end{continueexample}


\subsection{Critical degrees and recurrence relations.}
The explicit construction of an exceptional recurrence relation
\eqref{eq:sigqRlam} requires knowledge of the critical degrees $q$ of
$\cRslam$. For each such $q\in \cDslam$, one also requires the
eigenvalue $\sigma_q (x,y)$ and the sequence of polynomials
$\pinat_{q,k}(m),\; k=-q,-q+2,\ldots, q$, which serve as the
coefficients of the recurrence relation.  From an algorithmic
standpoint, the determination of $q, \sigma_q$ and the
$\pinat_{q,k}(m)$ is a combined calculation.  By Proposition
\ref{prop:pinatjgamlam},
\begin{equation}
  \label{eq:betadef}
  \pinat_{q,k}(m) = \gamlam_k(m)\alq_{k}(m),\quad k=-q,-q+2,\ldots, q.
\end{equation}
where the $\alq_k(m)$ are polynomials, with $\gamlam_k(m)$ fixed as
per \eqref{eq:gamqdef}. Thus, for a given $q \geq 1$, one has to
consider a certain homogeneous linear system whose unknowns are the
$q+1$ polynomials $\alq_k(m)$.  If the system has a non-trivial
solution, then the corresponding $q$ is a critical degree. One can
extract the eigenvalue, and the coefficients of the recurrence
relation from the corresponding solution.

By Lemma \ref{lem:degpigam}, $\deg \pinat_{q,k}(m)\le (q+k)/2$.
Hence, by \eqref{eq:gamqdef},
\begin{equation}
  \label{eq:bqkdef}
  \deg \alq_k(m) \le \bq_{k} := \frac12(q+k)-g_k,
\end{equation}
where $g_k = \glam_k= \deg \gamma_k(m)$ for notational convenience.
In other words, $\bq_{k}$ is an upper bound for the degrees of freedom
inherent in the choice of $\pinat_{q,k}(m)$.

We represent the level $q$ variables using the truncated list
\[ \balphaq = ( \alq_{-q}(m),\alq_{-q+2}(m),\ldots, \alq_q(m) ) \]
and set
\begin{equation}
  \label{eq:pinatbeta}
  \begin{aligned}
    \pinat(x,y,z;\balphaq) &= y^{q/2}\sum_{k\in \Z}\gamlam_{k}(m) \alq_{k}(m)
    y^{-k/2} z^{-k}.
  \end{aligned}
\end{equation}
As per Proposition \ref{prop:pinatzm}, let $\pi(x,y,z;\balphaq)$ be
such that
\begin{equation}
\label{eq:pipinatbeta}
\pi(\partial_z,y,z;\balpha) z^m = \pinat(m,y,z;\balpha) z^m.
\end{equation}
Set
\begin{align}
  \label{eq:hpibeta}
  \hpi(x,y,z;\balphaq)
  := \bH(y,z)\pi(x,y,z;\balphaq),\\ \nonumber
  =y^{q/2} \sum_{ik} \hpi_{ik}(\balphaq)x^iy^{-k/2} z^{i-k}\\
  \label{eq:sigbeta}
  \sigma(x,y;\balphaq) := y^{q/2} \sum_k \hpi_{kk}(\balphaq)x^iy^{-k/2} 
\end{align}
By Proposition \ref{prop:pi*crit},
$\pi(\partial_z,y,z;\balpha) \in \cSlam$ if and only if
\begin{equation}
  \label{eq:hpibeq}
  \hpi_{ik}(\balphaq) = 0,\quad  i>k 
\end{equation}
If that is the case, then
\[ \piflat(x,y,z;\balphaq) = \sigma(x,y;\balphaq).\] Thus,
\eqref{eq:hpibeq} constitutes the linear system for the recurrence
relations.

Let us write
\begin{equation}
  \label{eq:beqdef}
  \alq_{k}(m)  = \sum_{a=0}^{\bq_k} \al_{ka} m^{a},\quad k=-q,-q+2,\ldots, q
\end{equation}
where $\al_{ka}$ are lexicographically ordered indeterminates. This
means that $\al_{k_1a_1} \preceq \al_{k_2a_2}$ if and only if
$k_1<k_2$, or if $k_1=k_2$ and $a_1 \le a_2$.  Let us also say that
$\al_{k_1a_1}$ and $\al_{k_2 a_2}$ have the same parity if
$k_1\equiv k_2 (\bmod 2)$.  One can show that
\begin{equation}
  \label{eq:piijpivot}
  \hpi_{a+g_k,k}(\balphaq) = \al_{ka} + \ldots,\quad 
\end{equation}
where the $\ldots$ indicates terms of higher lexicographic order and
equal parity. Thus, the system \eqref{eq:hpibeq} is quasi-triangular,
because $\al_{ka}$ can be eliminated provided $a+g_k>k$.

Also, by \eqref{eq:alphasym} we have
\begin{equation}
  \label{eq:alphasym2}
  \al_{-ka} = (-2)^k\lp \al_{ka} + \sum_{i=a+1}^{\bq_k}
  \binom{i}{a} \al_{ki}  \rp,\quad a=0,\ldots, \bq_k.
\end{equation}
Thus, for $k<0$, the row-reduction may be improved by employing the
universal \eqref{eq:alphasym2} in place of the more computationally
demanding \eqref{eq:hpibeq}.

\begin{continueexample}{ex:L22}
  Let us determine the critical degrees for the partition
  $\lambda=(2,2,0,\ldots)$.
  The Maya diagram and its translates are shown in the figure below.
  The black-filled boxes belong to $\Mlam+q$, the empty boxes below to
  $\Jlam$; the red-filled boxes belong to $\Glam_q = (\Mlam+q)\cap \Jlam$.

  \begin{figure}[ht]
  \centering
  \begin{tikzpicture}[scale=0.5]
    
    \fill  (-4.5,8.5)
    \ncirc2 \ncirc1;

    \fill  (-4.5,7.5)
    \ncirc0 \ncirc1;
    \fill[red] (-4.5,7.5) \ncirc4\ncirc1;

    \fill  (-4.5,6.5)
    \ncirc0 \ncirc1\ncirc1\ncirc1 \ncirc3  \ncirc1;


    \fill  (-4.5,5.5)  \ncirc0 \ncirc1 \ncirc1\ncirc1 ;
    \fill[red]  (-4.5,5.5)   \ncirc4\ncirc1\ncirc3\ncirc1;



    \fill  (-4.5,4.5)    \ncirc0 \ncirc1 \ncirc1\ncirc1 \ncirc3\ncirc1;
    \fill[red]  (-4.5,4.5) \ncirc4\ncirc1\ncirc5\ncirc1;

    \fill  (-4.5,3.5)    \ncirc0 \ncirc1 \ncirc1\ncirc1 \ncirc3\ncirc1;
    \fill[red]  (-4.5,3.5) \ncirc4\ncirc1\ncirc3\ncirc3\ncirc1;

    \fill  (-4.5,2.5)    \ncirc0 \ncirc1 \ncirc1\ncirc1 \ncirc3\ncirc1;
    \fill[red]  (-4.5,2.5) \ncirc4\ncirc1\ncirc3\ncirc1\ncirc3\ncirc1;

    \path (10.5,9.5) node[anchor=west] {$\Mlam-6$};
    \path (10.5,8.5) node[anchor=west] {$\Mlam-4$};
    \path (10.5,7.5) node[anchor=west] {$\Mlam-2$};
    \path (10.5,6.5) node[anchor=west] {$\Mlam$};
    \path (10.5,5.5) node[anchor=west] {$\Mlam+2$};
    \path (10.5,4.5) node[anchor=west] {$\Mlam+4$};
    \path (10.5,3.5) node[anchor=west] {$\Mlam+5$};
    \path (10.5,2.5) node[anchor=west] {$\Mlam+6$};

    \foreach \y in {2.5,...,9.5} \fill (-6.5,\y) \ncirc0\ncirc1;

    \draw  (-7,2) grid +(16 ,8);
    \draw[line width=2pt] (-5,2) -- ++ (0,8);

  \foreach \x in {-8,...,7} \draw (\x+1.5,1.5)  node {$\x$};

\end{tikzpicture}
\caption{Translates of $\Mlam$ where $\lambda = (2,2,0,\ldots)$.}
\end{figure}

The critical degrees of $\cRlam$ are the shifts $q$ for which
$\Mlam\subset \Mlam+q$.  These are also the shifts for which
$\glam_q =\#\Glam_q=q$.  The above table indicates that the set of all
such shifts is $\cDlam = \{ 0,4,5,6,\ldots \}$.  These are also the
orders of the lowering operators for this partition.  Not all of these
are critical degrees of $\cRslam$.  Since $\taulam_y(x) = x^4+12y^2$
has simple zeros for $y\neq 0$, Proposition \ref{prop:tausimp} may be
applied to conclude that $\cDslam = \{5,6,7,\ldots \}$.  This can also
be established using a direct calculation using criterion
\eqref{eq:hpibeq}.

We now illustrate the relevant procedure by determining the recurrence
relation for $q=6$.  By \eqref{eq:pinatbeta}, the generic operator
that preserves $\Wlam$ and has shifts $-6,-4,\ldots, 6$ is given by
  \begin{align*}
    \pinat(m,y,z,\balpha^{(6)})
    &= \al_{-6,0} y^6 z^6 + (\al_{-4,0} + \al_{-4,1} m) y^5 z^4
      + \al_{-2,0} 
      (m+2)(m+1) y^4 z^2 \\
    &\qquad + (\al_{00} + \al_{01}m + \al_{02}m^2+
      \al_{03}m^3)y^3\\
    &\qquad + \al_{20} (m+2)(m+1)(m-2)(m-3) y^2z^{-2}\\
    &\qquad + (\al_{40}+m \al_{41}) (m+2)(m+1)(m-4)(m-5) y z^{-4}\\
    &\qquad + \al_{60} (m+2)(m+1)(m-2)(m-3)(m-6)(m-7) z^{-6}\\
  \end{align*}
  By \eqref{eq:hpibeq}, we will have $\pi\in \cSlam$ provided
  $\hpi_{ij}(\balpha^{(6)}) = 0$ for all $i>j$. In that case, by
  \eqref{eq:sigbeta},
  \begin{align*}
    \sigma(x,y;\balpha^{(6)})
    &= (-6x^2 y^2  -12y^3)\al_{20} +( x^4y+12 x^2 y^2-52 y^3)
      \al_{40}\\
    &\qquad +    (4 x^4y-48x^2 y^2-144   y^3)\al_{41} \\
    &\qquad + ( x^6+30 x^4 y-396 x^2 y^2-264   y^3)\al_{60} 
  \end{align*}
  will be the corresponding eigenvalue.

  Applying \eqref{eq:pipinatbeta} and \eqref{eq:hpibeta}, the linear
  system in question has the following matrix:
  \[ \begin{array}{c|ccccccc}
       \hpi_{ij}(\balpha^{(6)}) &\al_{01}& \al_{02}& \al_{03}&
      \al_{20}&\al_{41}&\al_{40} &\al_{60} \\\hline
       \hpi_{32}&  0 & 0 & 0 & 4 & 52 & 8 & 240 \\
       \hpi_{10}&   1 & 1 & 1 & 0 & 60 & 48 & -48 \\
       \hpi_{20}& 0 & 1 & 3 & 36 & 216 & 24 & 720 \\
       \hpi_{30}& 0 & 0 & 1 & 8 & 40 & 0 & 160 \\
       \hpi_{42}& 0 & 0 & 0 & 1 & 10 & 0 & 60 \\
       \hpi_{54}& 0 & 0 & 0 & 0 & 1 & 0 & 12
     \end{array}
   \]
   The symmetry relations \eqref{eq:alphasym2} give
   \[ \al_{-2,0}=4 \al_{2,0},\;\al_{-4,0}=16 (\beta _{4,0}+4
     \al_{4,1}),\;\al_{-4,1}=16 \al_{4,1},\;\al_{-6,0}=64 \al_{6,0} \]
   Setting $\al_{0,0}=0,\al_{6,0}=1$, solving the above relations,
   and using \eqref{eq:Thetlamqpinat} gives the following recurrence
   relation of order $12$:
   \begin{align*}
     (x^6+&36x^2y^2-192y^3)\Rlam_m\\
          &=R_{m+6}-6 (2 m+5) y R_{m+4}+60 (m+1) (m+2) y^2 R_{m+2}
            +(304 m-240 m^2-160 m^3) y^3     R_m\\
          &\quad +240 (m-3) (m-2) (m+1) (m+2) y^4 R_{m-2}\\
          &\quad -96 (m-5) (m-4)       (m+1) (m+2) (2 m-3) y^5 R_{m-4}\\
          &\quad +64 (m-7) 
            (m-6) (m-3) (m-2) (m+1) (m+2) y^6 R_{m-6}
   \end{align*}
   This corresponds to the eigenvalue equation
   \[ \pi_6(\partial_z,y,z)\Psilam(x,y,z) =
     (x^6+36x^2y^2-192y^3)\Psilam(x,y,z) ,\]
   where
   \begin{align*}
     \pi_6(\partial_z,y,z)
     &= \partial_z^6-12 y z \partial_z^5+
       \left(60 y^2 z^2-30 y-24z^{-2}\right) \partial^4_{z}\\
     &+\left(-160 y^3 z^3+240 y^2
       z+192 yz^{-1}+96z^{-3}\right) \partial^3_{z}+\\
     &\quad +\left(240 y^4 z^4-720 y^3 z^2-360
       y^2-288 yz^{-2}-108z^{-4}\right) \partial_z^2\\
     &\quad+\left(-192 y^5 z^5+960 y^4 z^3-96 y^3 z-288yz^{-3} -
       144z^{-5}\right) \partial_z\\ 
     &\quad + \left(64 y^6 z^6-480 y^5 z^4+480 y^4
       z^2+720 y^2z^{-2}+720 yz^{-4}+504z^{-6}\right)
   \end{align*}

   Let us consider the similar calculation for $q=4$. Generically,
   \begin{align*}
    \pinat(m,y,z,\balpha^{(4)})
    &= \al_{-4,0}  y^4 z^4
      + (\al_{00} + \al_{01}m + \al_{02}m^2)y^2\\
    &\quad + \al_{40} (m+2)(m+1)(m-4)(m-5) y z^{-4}
   \end{align*}
   The linear system $\hpi_{ik}(\balpha^{(4)})=0,\; i>k$ is a
   truncation of the $\balpha^{(6)}$ system shown above.  The
   corresponding matrix
  \[ \begin{array}{c|ccccccc}
       \hpi_{ij}(\balpha^{(4)}) &\al_{01}& \al_{02}&\al_{40}  \\\hline
       \hpi_{32}&  0 & 0 & 8 \\
       \hpi_{10}&   1 & 1  & 48 \\
       \hpi_{20}& 0 & 1 & 24 
     \end{array}
   \]
   has maximal rank, which means that $4\notin \cDslam$.  In other
   words, just as predicted by Proposition \ref{prop:tausimp}, there
   is no recurrence relation of order $8$.

 \end{continueexample}

\section{Conclusions and Remarks}

Both the wave functions in the adelic Grassmannian and the exceptional
Hermite polynomials exhibit bispectrality.  However, it was not
previously recognized that some of those wave functions were
generating functions for the exceptional Hermites.  That this
fundamental connection previously went unnoticed may be a consequence
of the fact that Wilson's bispectral wave functions were obtained by
setting all higher KP variables $t_i$ for $i>1$ to zero while this
correspondence holds only when $y=t_2$ is non-zero.

Stating the correspondence precisely required the use of new notation
and some technical lemmas.  It is also stated most naturally not in
terms of the exceptional Hermite \textit{polynomials} $\Hlam_n$ but
rather through their rational counterparts, $\Rlam_m$. Nevertheless,
the rewards are worth these efforts.  Many of the known properties of
the exceptional Hermites are easily rederived from the bispectrality
of these generating functions.  Moreover, utilizing this connection
also leads to new results and more effective algorithms for computing
the associated algebras of operators.

One of the key benefits to situating exceptional polynomials within
$\Grad$ is the realization that there are two relevant notions of
bispectrality: differential-differential and
differential-difference. A consequence of this remark is the existence
of a difference intertwiner that serves to give exceptional
polynomials as a canonical linear combination of their classical
counter-parts.  The other consequence, of course, is the
re-interpretation of exceptional recurrence relations in terms of the
commutative algebra of operators canonically associated to every point
in $\Grad$.

This paper considered only the wave functions associated to a
collection of points in $\Grad$ indexed by partitions and their flows
under the second flow of the KP hierarchy.  These wave functions are
precisely the generating functions for the exceptional Hermite
functions with a scaling parameter.  It is our intention to consider
in a future paper how this construction generalizes to other points in
the adelic Grassmannian and to their dependences on the higher KP time
variables.


\begin{thebibliography}{99}
\bibitem{bakalov} B. Bakalov, E. Horozov, and M. Yakimov, General
  methods for constructing bispectral operators. \textit{Physics
    Letters A} \textbf{222} 59-66 (1996)

\bibitem{bochner}
S.~Bochner, \emph{{\"Uber Sturm-Liouvillesche Polynomsysteme}}, Mathematische
  Zeitschrift \textbf{29} (1929), 730--736.


\bibitem{bondunstev} N. Bonneux, C. Dunning, M. Stevens, Coefficients
  of Wronskian Hermite polynomials, \textit{Studies in Applied
    Mathematics} 2019.
\bibitem{DG} J.J. Duistermaat and F.A. Gr\"unbaum, Differential
  Equations in the Spectral Parameter. \textit{Commun. Math. Phys.}
  \textbf{103}, 177-240 (1986)
\bibitem{duran} A.J. Dur\'an, Higher order recurrence relation for
  exceptional Charlier, Meixner, Hermite and Laguerre orthogonal
  polynomials, Integral Transforms Spec. Funct., 26 (2015), 357–376.
\bibitem{duran2} A.J. Dur\'an, Exceptional Charlier and Hermite
  polynomials. J. Approx. Theory 182 (2014), 29–58.
\bibitem{GGM}  {David G\'omez-Ullate, Yves Grandati, and Robert Milson, Rational extensions of the quantum harmonic oscillator and exceptional Hermite polynomials, \textit{J. Phys. A} \textbf{47} (2013), no. 1, 015203.}

\bibitem{GFGM} M.A. Garc\'ia-Ferrero,D. G\'omez-Ullate, and Robert
  Milson. A Bochner type characterization theorem for exceptional
  orthogonal polynomials. J. of Mathematical Analysis and Applications
  \textbf{472} (2019): 584-626.
  


\bibitem{GKKM} { D. G\'{o}mez-Ullate, A. Kasman,
  A.B. Kuijlaars, and R.  Milson, Recurrence relations for
  exceptional {H}ermite polynomials, \textit{J. Approx. Theory}
  \textbf{204} (2016) 1--16.}


\bibitem{GGMM} D. G\'omez-Ullate, Y. Grandati, Z. McIntyre, and
  R. Milson, Ladder operators and rational extensions. arXiv preprint
  arXiv:1910.12648, 2019.

\bibitem{GKM} D.~G\'{o}mez-Ullate, N.~Kamran, and R.~Milson, \emph{{An
      extended class of orthogonal polynomials defined by a
      Sturm-Liouville problem}}, Journal of Mathematical Analysis and
  Applications \textbf{359} (2009), no.~1, 352--367.

\bibitem{GH} F.A. Gr\"unbaum and L. Haine, Orthogonal polynomials
  satisfying differential equations: the role of the Darboux
  transformation. CRM Proceedings and Lecture Notes. Vol. 9. AMS 1996.
  
\bibitem{GY} F.A. Gr\"unbaum and M. Yakimov, Discrete bispectral
  Darboux transformations from Jacobi operators, \textit{Pacific
    Journal of Mathematics} \textbf{204} (2002)  2 395--431.

\bibitem{HI} L. Haine, and P. Iliev. Commutative rings of difference
  operators and an adelic flag manifold. International Mathematics
  Research Notices \textbf{2000.6} (2000): 281-323.

\bibitem{iliev} P. Iliev, Bispectral extensions of the Askey-Wilson
  polynomials. J. Functional Analysis \textbf{266} (2014): 2294-2318.
  
\bibitem{KRAiry} A. Kasman and M. Rothstein, Bispectral Darboux
  Transformations: the Generalized Airy Case \textit{Physica D}
  \textbf{102} (1997), no. 3-4 pp. 159- 176.

  
\bibitem{macdonald} Macdonald IG. \textit{Symmetric functions and Hall
    polynomials}. Oxford university press, 1998.
\bibitem{noumi} Noumi M. \textit{Painlev\'e equations through
    symmetry.} Vol. 223. Springer Science \& Business, 2004.


  




\bibitem{oblomkov} A.~A. Oblomkov, \emph{{Monodromy-free Schr\"odinger
      operators with quadratically increasing potentials}},
  Theoretical and Mathematical Physics \textbf{121} (1999), no.~3,
  374--386.


\bibitem{OS1}
S.~Odake and R.~Sasaki, \emph{{Infinitely many shape invariant potentials and
  new orthogonal polynomials}}, Physics Letters B \textbf{679} (2009), no.~4,
  414--417.

\bibitem{OS2}
S.~Odake, \emph{{Recurrence relations of the multi-indexed orthogonal
  polynomials}}, Journal of Mathematical Physics \textbf{54} (2013), no.~8,
  083506.
  

\bibitem{RKO} {Gian-Carlo Rota, David Kahaner, and Andrew Odlyzko. On
    the foundations of combinatorial theory. VIII. Finite operator
    calculus. Journal of Mathematical Analysis and Applications
    \textbf{42.3} (1973), 684-760.}

\bibitem{Sato} M. Sato and Y. Sato, in Nonlinear partial differential equations in applied science (Tokyo,
  1982), 259–271, North-Holland, Amsterdam, 1983
  

\bibitem{SW} G. Segal and G. Wilson, Loop Groups and Equations of
    KdV Type. \textit{Publications Mathematiques} \textbf{61} de
    l'lnstitut des Hautes Etudes Scientifiques, 5-65 (1985).


\bibitem{wilson} G. Wilson, Bispectral commutative ordinary
  differential operators. \textit{J. reine angew. Math} \textbf{442}
  pp. 177-204, (1993)

\end{thebibliography}
\end{document}